\documentclass[english,12pt,a4paper]{article}
\usepackage{amsmath,amssymb,amsthm} 

\usepackage{graphicx}
\usepackage[scanall]{psfrag}

\newtheorem{theorem}{Theorem}
\newtheorem{lemma}[theorem]{Lemma}
\newtheorem{proposition}[theorem]{Proposition}
\newtheorem{definition}[theorem]{Definition}
\newtheorem{corollary}[theorem]{Corollary}

\newtheorem{assumptions}{Standing Assumption}

\numberwithin{theorem}{section}
\numberwithin{equation}{section}

\newcommand{\R}{{\mathbb{R}}}
\newcommand{\rplus}{{\mathbb{R}^{+}}}
\newcommand{\rpluscl}{\overline{\mathbb{R}}^{+}}

\newcommand{\C}{\mathbb{C}}

\newcommand{\dif}[0]{\ensuremath{\,\mathrm{d}}}

\newcommand{\clos}[1]{\overline{#1}}

\renewcommand{\Re}{{\rm Re\,}}

\newcommand{\AB}     {{A\& B}}                  
\newcommand{\CD}     {{C\& D}}                  

\newcommand{\Xscr} {{\mathcal X}}
\newcommand{\Yscr} {{\mathcal Y}}
\newcommand{\Uscr} {{\mathcal U}}
\newcommand{\Zscr} {{\mathcal Z}}
\newcommand{\Escr} {{\mathcal E}}

\newcommand{\bbm}[1]{\left[\begin{matrix} #1 \end{matrix}\right]}
\newcommand{\sbm}[1]{\left[\begin{smallmatrix} #1
             \end{smallmatrix}\right]}

\newcommand{\SysNode}{\bbm{\AB \cr \CD}}

\newcommand{\SmallSysNode}{\sbm{\AB \cr \CD}}

\newcommand{\DisplayNote}[1]{\hspace{.5cm}( {\bf{#1}} ) \hspace{.5cm} }
\newcommand{\OmitThis}[1]{}
\newcommand{\ResearchNote}[1]{}

\newcommand{\ProofNote}[1]{}

\newcommand{\abs}[1]{\left \vert #1 \right \vert}

\newcommand{\Dom}[1]{{\rm dom}\left (#1 \right )}

\newcommand{\Null}[1]{{\rm ker} \left (#1 \right )}
\newcommand{\BLO}{\mathcal L}
\newcommand{\norm}[1]{\|{#1}\|}
\newcommand{\rst}[1]{\big{|}_{#1}}

\newcommand{\bg}{\mathbf \gamma}
\newcommand{\bl}{\mathbf l}
\newcommand{\bt}{\mathbf t}
\newcommand{\bn}{\mathbf n}
\newcommand{\bb}{\mathbf b}
\newcommand{\bnu}{\mathbf \nu}

\newcommand{\br}{\mathbf r}

\newcommand{\Dscr} {{\mathcal D}}

\begin{document}
\title{\emph{A posteriori} error estimates \\ for Webster's equation
  \\ in wave propagation} \author{Teemu Lukkari and Jarmo Malinen}

\bibliographystyle{plain}

\maketitle

\begin{abstract}
  We consider a generalised Webster's equation for describing wave
  propagation in curved tubular structures such as variable diameter
  acoustic wave guides. Webster's equation in generalised form has
  been rigorously derived in a previous article starting from the wave
  equation, and it approximates cross-sectional averages of the
  propagating wave. Here, the approximation error is estimated by an
  \emph{a posteriori} technique.
\end{abstract}

\noindent {\bf Keywords.} Wave propagation, tubular domain, Webster's model, a posteriori error analysis.

\noindent {\bf AMS classification.} Primary  37L05. Secondary 35L05, 35L20, 47N70, 93C20.


\section{\label{IntroSec} Introduction}

We study wave propagation in a narrow but long, tubular domain $\Omega
\subset \R^3$ of finite length whose cross-sections are circular and
of varying area.  In this case, the wave equation in the domain
$\Omega$, i.e., the topmost equation in \eqref{IntroWaveEq} below, has
a classical approximation depending on a single spatial variable in
the long direction of tubular $\Omega$. The approximation is known as
\emph{Webster's equation}, which is given in generalised form as the
topmost equation in \eqref{IntroWebstersEq} below. The geometry of
$\Omega$ is represented by the \emph{area function} $A(\cdot)$ whose
values are cross-sectional areas of $\Omega$. The solution of
Webster's equation approximates cross-sectional averages of the
solution to the wave equation as shown in \cite{L-M:WECAD}.  The
purpose of this article is to estimate the approximation error by an
\emph{a posteriori} method, using the passivity and well-posedness
estimates given in \cite{A-L-M:AWGIDDS} as well as analytic tools
presented in \cite[Section~5]{L-M:WECAD}.

Webster's original work \cite{AW:AITHP} was published in 1919, but the
model itself has a history spanning over 200 years and starting from
the works of D.~Bernoulli, Euler, and Lagrange. Early work concerning
Webster's equation can be found in
\cite{EE:CSWHE,VS:GPWHT,VS:NFH,AW:AITHP}, and a selection of
contemporary approaches is provided by
\cite{L-L:AMAEMAI,L-L:AMAEMAII,N-T:APDVCS,SR:STSVCALDF,SR:WHER,R-E:NCBMSFESSPLFD}
and, in particular, \cite{R-H:IA}.  The derivation of Webster's
equation in \cite{SR:WHER} (see also \cite{AN:PM}) is based on
asymptotic expansions that, however, does not give estimates for the
approximation error. The resonance structure of Webster's equation is
obtained from the associated eigenvalue problem which resembles the
characterisation for the asymptotic spectra of Neumann--Laplacian on
shrinking tubular domains in \cite{KZ:CSMSCG,RS:STRIP}.  This is an
example of dimensional reduction that is also the basis of shell and
plate models; see, e.g., \cite{A-A-F-M:DJPMVM} where the treatment is
for the stationary problems, only.  Similarly, strings have been
considered in \cite{A-B-P:VDSEES} where the tool for dimensional
reduction is the $\Gamma$-convergence of energy functionals as opposed
to starting from a partial differential equation. In our approach, the
dimensional reduction is based on the wave equation, and it is carried
out by averaging over those degrees of freedom that are not part of
Webster's equation; see \cite{L-M:WECAD}.

Our interest in Webster's equation stems from the fact that it
provides a model for the acoustics of the human vocal tract as it
appears during a vowel utterance.  Webster's equation can be used as a
part of a dynamical computational physics model of speech as discussed
in
\cite{A-A-H-J-K-K-L-M-M-P-S-V:LSDASMRIS,C-K:TVINAS,GF:ATSP,H-L-M-P:WFWE}
and the theses \cite{A-A:LOGMNFMCAL,T-M:MVP}.  Further applications of
Webster's equation include modelling of water waves in tapered
channels, acoustic design of exhaust pipes and jet engines for
controlling noise, vibration, and performance as well as construction
of instruments such as loudspeakers and horns
\cite[p.~402--405]{H-S:FDHL}.

The results of this article describe the interplay between two kinds
of models for acoustic waveguides; i.e., wave equation and Webster's
equation. The first of the models is suitable for high precision, and
the latter is computationally more efficient but lacks, e.g.,
transversal wave propagation because of simplifications. The two
models are related to each other by the common underlying geometry of
the waveguide.  The waveguide geometry is originally defined by the
tubular domain $\Omega \subset \R^3$ that has the following
properties.  The centreline of the tube is a smooth planar curve $\bg$
of unit length and with vanishing torsion, parametrised by its arc
length $s \in [0,1]$. We assume that the cross-section of $\Omega$,
perpendicular to the tangent of $\bg$ at the point $\bg(s)$, is the
circular disk $\Gamma(s)$ with centre point $\bg(s)$. The radius of
$\Gamma(s)$ is denoted by $R(s)$ with area $A(s)$. The boundary
$\partial \Omega$ of $\Omega$ consists of the \emph{ends} of the tube,
$\Gamma(0)$ and $\Gamma(1)$, and the \emph{wall} $\Gamma := \cup_{s
  \in [0,1]}{\partial \Gamma(s)}$ of the tube.

With this notation, acoustic wave propagation in $\Omega$ can be
modelled by the wave equation, written for the \emph{(perturbation)
  velocity potential} $\phi:  \rpluscl \times \Omega \to \R$
\begin{equation}\label{IntroWaveEq}
\begin{cases}
  &  \phi_{tt}(t,\br) = c^2 \Delta \phi(t,\br) \quad
  \text{ for } \br \in \Omega \text{ and } t \in \rplus, \\
  &   c \frac{\partial \phi}{\partial
    \bnu}(t,\br) + \phi_t(t,\br) = 2 \sqrt{\tfrac{c}{\rho A(0)}} \, u(t,\br) \quad
  \text{ for } \br \in \Gamma(0) \text{ and }  t \in \rplus, \\
  &  \phi(t,\br)  = 0 \quad \text{ for } \br \in \Gamma(1) \text{ and }
  t \in \rplus, \\
  &  \frac{\partial \phi}{\partial \bnu}(t,\br) + \alpha \phi_t (t,\br)   = 0 
  \quad \text{ for } \br \in \Gamma, \text{ and }
  t \in \rplus, \text{ and } \\
  & \phi(0,\br)  = \phi_0(\br), \quad \rho \phi_t(0,\br) = p_0(\br) \quad \text{ for
  } \br \in \Omega
\end{cases}
\end{equation}
with the observation defined by
\begin{equation}\label{IntroWaveObs}
  c \frac{\partial \phi}{\partial
    \bnu} (t,\br) - \phi_t(t,\br)  =  2 \sqrt{\tfrac{c}{\rho A(0)}} \, 
  y(t,\br)  \quad
  \text{ for }  \br \in \Gamma(0) \text{ and }  t \in \rplus,
\end{equation}
where $\rplus = (0,\infty)$, $\rpluscl = [0,\infty)$, $\bnu$ denotes
  the unit normal vector on $\partial \Omega$, $c$ is the sound speed,
  $\rho$ is the density of the medium, and $\alpha \geq 0$ is a
  parameter associated to boundary dissipation.  The Dirichlet
  condition on $\Gamma(1)$ represents an open end, and the Neumann
  condition on $\Gamma$ represents a hard reflective surface.  The
  control (i.e., the input) $u(t,\br)$ and the observation (i.e., the
  output) $y(t,\br)$ are given in \emph{scattering form} in
  \eqref{IntroWaveEq} where the physical dimension of both signals is
  power per unit area.

It was shown in \cite[Theorem~5.1 and Corollary~5.2]{A-L-M:AWGIDDS}
that for \linebreak $u \in C^2(\rpluscl;L^2(\Gamma(0)))$ and the initial state
$\sbm{\phi_0 \\ p_0}$ compatible with the input $u$ (as detailed below
in Assumption~\eqref{MainTheorem2CompatibilityAss} of
Theorem~\ref{MainTheorem2}), there exists a unique classical solution
$\phi$ of \eqref{IntroWaveEq} satisfying
\begin{equation} \label{WaveEqSolutionRegular}
\begin{aligned}
    & \phi \in C^1(\rpluscl; H^1(\Omega)) \cap C^2(\rpluscl;
  L^2(\Omega)), \\ & \nabla \phi \in C^1(\rpluscl;L^2(\Omega;\R^3)),
  \text{ and } \Delta \phi \in C(\rpluscl;L^2(\Omega)).
  \end{aligned}
\end{equation}
Then the function $y$ given by \eqref{IntroWaveObs} satisfies $y \in
C(\rpluscl;L^2(\Gamma(0)))$.  For the rest of this article, $u$,
$\phi$, and $y$ always denote these functions.

Following \cite{L-M:WECAD}, the generalised Webster's equation for
the velocity potential $\psi:\rpluscl \times [0,1]  \to \R$ is given by
\begin{equation} \label{IntroWebstersEq}
  \begin{cases}
    & \psi_{tt} = \frac{c(s)^{2}}{A(s)} \frac{\partial}{\partial s}
    \left ( A(s) \frac{\partial \psi }{\partial s} \right ) 
    - \frac{2 \pi \alpha W(s) c(s)^{2} }{A(s)} \psi_t
\\   & \hfill \text{ for }   s \in (0,1) \text{ and } t \in \rplus, \\
    & - c \psi_s(t,0) + \psi_t(t,0) = 2 \sqrt{\frac{c}{\rho
        A(0)}} \, \tilde u(t) \quad
    \text{ for } t \in \rplus,  \\
    & \psi(t,1) = 0 \quad  \text{ for } t \in \rplus, \quad \text{ and } \\
    & \psi(0,s) = \psi_0(s), \quad \rho \psi_t(0,s) = \pi_0(s) \quad
    \text{ for } s \in (0,1),
\end{cases}
\end{equation}
and the observation $\tilde y$ is defined by
\begin{equation}\label{IntroWebstersEqObs}
   - c \psi_s(t,0) - \psi_t(t,0) = 2 \sqrt{\frac{c}{\rho
        A(0)}} \, \tilde y(t) \quad
    \text{ for } t \in \rplus.
\end{equation}
The constants $c$, $\rho$, $\alpha$ are same as in
\eqref{IntroWaveEq}, and $A(s)$ is the area of the cross-section
$\Gamma(s)$. Note that the dissipative boundary condition in
\eqref{IntroWaveEq} gives rise to a dissipation term in
\eqref{IntroWebstersEq}.  The \emph{stretching factor} is the function
$W(s) := R(s)\sqrt{R'(s)^2+(\eta(s) - 1)^2}$ where the \emph{curvature
  ratio} is given by $\eta(s) := R(s) \kappa(s)$ and $\kappa$ denotes
the curvature of the centreline $\bg$.  Because of the curvature of
$\Omega$, we adjust the sound speed for \eqref{IntroWebstersEq} by
defining $c(s) := c \Sigma(s)$ where $\Sigma(s) := \left ( 1 +
\tfrac{1}{4} \eta(s)^2 \right )^{-1/2}$ is the \emph{sound speed
  correction factor} as introduced\footnote{For generalised Webster's
  equation, we use the functions $A$, $\Sigma$, $\Xi$, $E$, and $W$
  that are introduced in terms of the tubular domain $\Omega$ in
  \cite{L-M:WECAD}. } in \cite[Section~3]{L-M:WECAD}.

\begin{assumptions} \label{StandingAssumption-1}
 We require that 
\begin{enumerate}
\item the tubular domain $\Omega$ does not fold into itself; i.e.,
  $\eta(s) < 1$ for all $s \in [0,1]$; and
\item the centreline $\bg(\cdot)$ and the radius function $R(\cdot)$
  are infinitely differentiable on $[0,1]$.
\end{enumerate}
\end{assumptions}
\noindent It follows from the smoothness that the rest of the data
satisfies
\begin{equation} \label{StandingSmoothness}
A(\cdot), \eta(\cdot), W(\cdot), c(\cdot), \Sigma(\cdot) \in
C^\infty([0,1]),
\end{equation}
and such a domain $\Omega$ satisfies all the assumptions listed in
\cite[Appendix~A]{A-L-M:AWGIDDS}.

The solution $\psi:[0,1]\times \rpluscl \to \R$ is \emph{Webster's
  velocity potential}.
It is expected to approximate the averages
\begin{equation} \label{AveragedWaveEqSol}
 \bar \phi(t,s) : = \frac{1}{A(s)} {\int_{ \Gamma(s) }{\phi d A}} \quad \text{ for } \quad 
s \in (0,1) \quad \text{ and }  \quad  t \in \rpluscl 
\end{equation}
of the velocity potential $\phi$ given by \eqref{IntroWaveEq} if the
inputs and initial states for both models are matched as shown in
Fig.~\ref{FigBoth}.  We call the difference $e := \psi - \bar \phi$
\emph{tracking error}, see the left panel of Fig.~\ref{FigBoth}. A
fundamental result on the tracking error is given in
\cite[Theorem~3.1]{L-M:WECAD}, and it is presented in right panel of
Fig.~\ref{FigBoth}: if the generalised Webster's equation is augmented
by an additional load function $f = F + G + H$, (depending on $\phi$
through \eqref{ControlTermsF}---\eqref{ControlTermsH} below), the
tracking error will vanish. We estimate the tracking error $e$ by a
method where the exact solution $\phi$ of the wave equation
\eqref{IntroWaveEq} is assumed to be known.  Hence, we call these
results \emph{a posteriori} estimates for Webster's equation even
though it is a solution of another equation that needs to be known.

\begin{figure}
  \includegraphics[height=5cm]{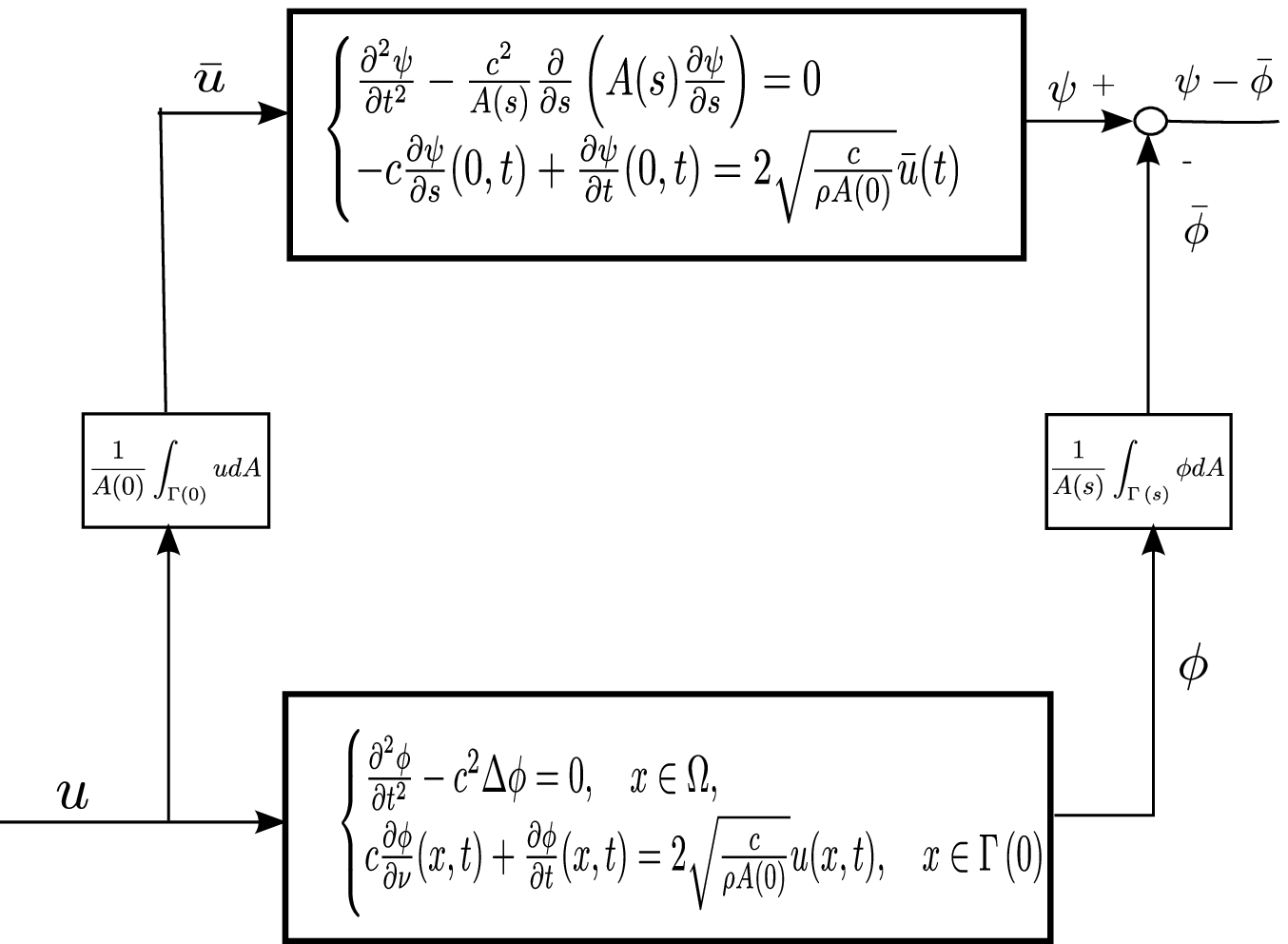}
  \includegraphics[height=5cm]{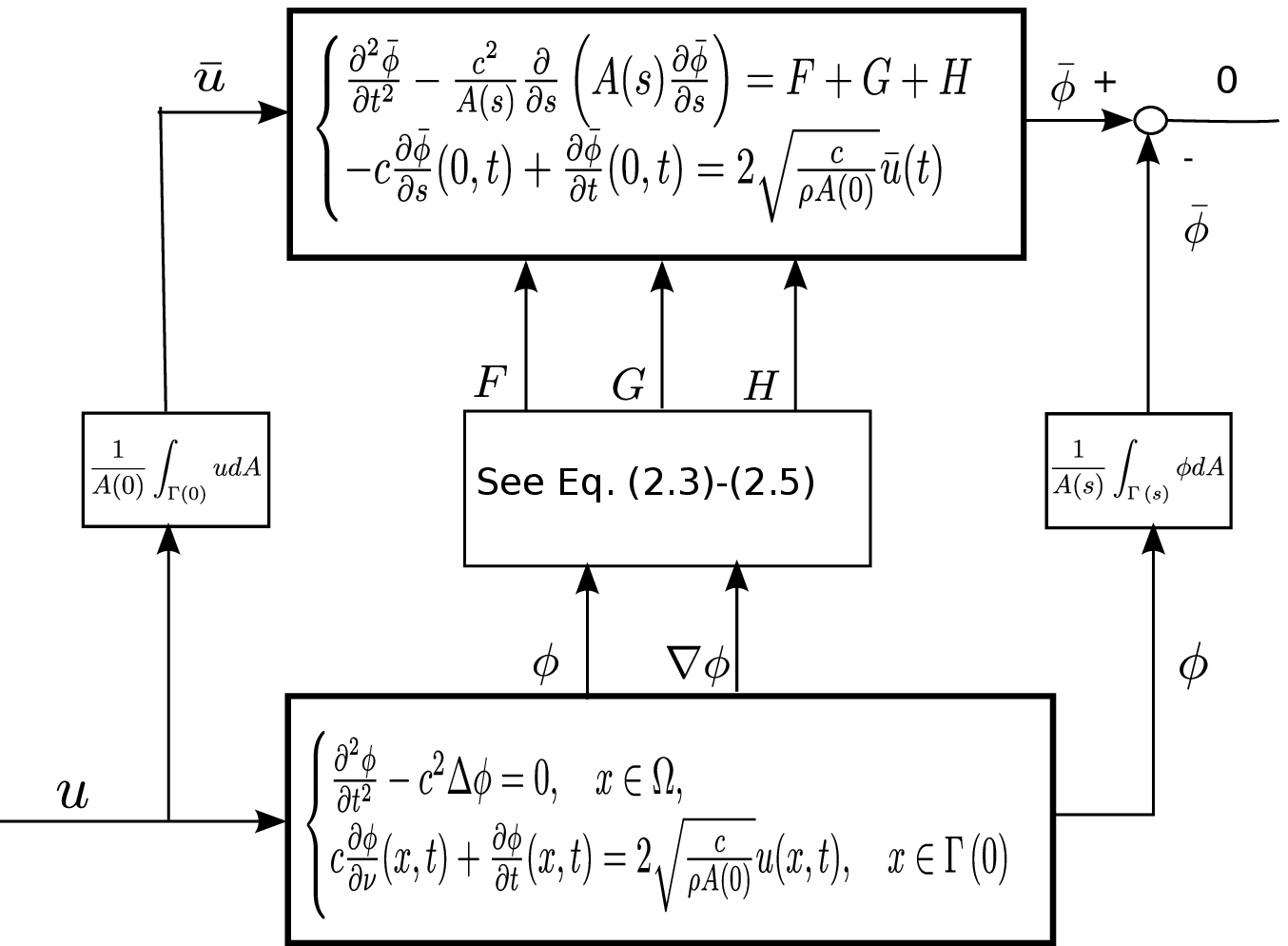} {Left panel: Feedforward coupling describing the
    tracking error $e = \psi - \bar \phi$. Right panel: The tracking
    error vanishes when the additional forcing functions $F, G$, and
    $H$ are applied. The equations in the blocks are as they
    appear in the lossless case $\alpha = 0$ and without curvature,
    i.e., $c(s) = c$. }
  \label{FigBoth}
\end{figure}


The article is organised as follows: we discuss the generalised
Webster's equation and its weak solution in the context of
\cite{L-M:WECAD} in Section~\ref{BackgroundSection} and also recall
the system node formulation from \cite{M-S-W:HTCCS}. We write the
inhomogeneous Webster's equation in terms of a scattering passive
system node and give the well-posedness estimate for the unique strong
solution in Section~\ref{WebsterDirectSec}.  This is used in the next
section where we show that the tracking error $e$ satisfies the first
\emph{a~posteriori} estimate, Theorem~\ref{MainTheorem2}. Then, we
estimate its right hand side by measuring how much $\phi$ differs from
its planar averages, leading to the second \emph{a~posteriori}
estimate, Theorem~\ref{MainTheorem3}.


\section{\label{BackgroundSection} Background}

\subsection{\label{InHomoWeb} Inhomogenous Webster's  equation}

Let us consider the interior/boundary point control problem
\begin{equation} \label{WebstersEqBnrCtrl}
  \begin{cases}
    & \psi_{tt} - \frac{c(s)^{2}}{A(s)} \frac{\partial}{\partial s}
    \left ( A(s) \frac{\partial \psi }{\partial s} \right ) 
    + \frac{2 \pi \alpha W(s) c(s)^{2} }{A(s)} \psi_t
= f(t,s)
\\   & \hfill \text{ for }   s \in (0,1) \text{ and } t \in \rplus, \\
    & - c \psi_s(t,0) + \psi_t(t,0) = 2 \sqrt{\frac{c}{\rho
        A(0)}} \, \tilde u(t) \quad
    \text{ for } t \in \rplus,  \\
    & \psi(t,1) = 0 \quad  \text{ for } t \in \rplus, \quad \text{ and } \\
    & \psi(0,s) = \psi_0(s), \quad \rho \psi_t(0,s) = \pi_0(s) \quad
    \text{ for } s \in (0,1),
\end{cases}
\end{equation}
with the observation $\tilde y$ is defined by
\begin{equation}\label{WebstersEqBnrCtrlObs}
   - c \psi_s(t,0) - \psi_t(t,0) = 2 \sqrt{\frac{c}{\rho
        A(0)}} \, \tilde y(t) \quad
    \text{ for } t \in \rplus.
\end{equation}
We allow for a nonvanishing load function $f$ in
\eqref{WebstersEqBnrCtrl}.  The reason for this is the fact that the
spatial averages $\bar \phi$ of $\phi$, given by
\eqref{AveragedWaveEqSol}, satisfy \eqref{WebstersEqBnrCtrl} (with
properly matched initial states and boundary control) as shown in
\cite[Theorem~3.1]{L-M:WECAD} if $f = F + G + H \in
C(\rpluscl;L^2(0,1))$ where
\begin{align} \label{ControlTermsF}
  F(t,s) & := - \frac{1}{A(s)} \frac{\partial}{\partial s} \left (
  A'(s) \left ( \bar \phi(s) - \frac{1}{2 \pi}\int_{0}^{2 \pi}
  \phi(s,R(s),\theta) \, d \theta \right ) \right ); \\
\label{ControlTermsG}
  G(t,s) & :=  - \frac{2 \pi  \alpha W(s)}{A(s)} \frac{\partial }{\partial t} \left ( \bar \phi(s) - \frac{1}{2
      \pi}\int_{0}^{2\pi}{ \phi(s,R(s),\theta)
      d\theta } \right ); \text{ and } \\
 \label{ControlTermsH}
  H(t,s) & := \int_{\Gamma(s)}{\frac{1}{\Xi} \nabla \left (
    \frac{1}{\Xi} \right )\cdot \nabla \phi \, d A } - \frac{1}{ A(s)}
  \int_{ \Gamma(s) } {E \Delta \phi d A } \\ & - \frac{\alpha W(s)
    \eta(s)}{A(s)} \left ( \int_{0}^{2\pi}{ \frac{\partial
      \phi}{\partial t} (s,R(s),\theta) \cos{\theta} d\theta } \right
  ) \nonumber
\end{align}
where the \emph{curvature factor} is given by $\Xi^{-1} :=  1 - r
  \kappa(s) \cos{\theta}$, and the \emph{error function} by
\begin{equation} \label{SSCFError}
   E(s,r,\theta) := \Xi^{-2} - \Sigma(s)^{-2} = - 2 r \kappa(s)
   \cos{\theta} + \kappa(s)^2 ( r^2 \cos^2{\theta} - R(s)^2/4);
\end{equation}
see \cite{L-M:WECAD} for details. It follows from the assumed
smoothness of $\gamma$ and $R(\cdot)$ and from
$\norm{\eta}_{L^\infty([0,1])} < 1$ that $E(\cdot), \Xi(\cdot) \in
C^\infty(\Omega)$.

In addition to the regularity \eqref{StandingSmoothness} of the
coefficient data for the Webster's model \eqref{WebstersEqBnrCtrl}, we
make additional requirements on the geometry of $\Omega$:
\begin{assumptions} \label{StandingAssumption}
We require that 
\begin{equation} \label{FurtherStandingAss}
\begin{aligned}
  & 0 < \min_{s \in [0,1]}{A(s)} \leq \max_{s \in [0,1]}{A(s)} < \infty \quad \text{ and } \\
  & 0 < \min_{s \in [0,1]}{c(s)} \leq \max_{s \in [0,1]}{c(s)} < \infty
\end{aligned}
\end{equation}
as well as $ A'(0)= \kappa(0) = 0$ at the control end $\Gamma(0)$ of
$\Omega$.
\end{assumptions}

We proceed to write \eqref{WebstersEqBnrCtrl} in operator form.
Define $W := \frac{1}{A(s)} \frac{\partial}{\partial s} \left ( A(s)
\frac{\partial}{\partial s} \right )$ and $D := - \frac{2 \pi W(s)
}{A(s)}$. Then the first of equations in \eqref{WebstersEqBnrCtrl} can
be cast into first order form by using the rule
\begin{equation*}
   \psi_{tt} = c(s)^{2} \left ( W \psi + \alpha D \psi_t \right )  + f \quad \hat{=} \quad
  \frac{d}{dt} \bbm{\psi \\ \pi} 
  = \bbm{0 & \rho^{-1} \\  \rho c(s)^2 W & \alpha c(s)^2 D} \bbm{\psi \\ \pi} + \bbm{0 \\ \rho f}.
\end{equation*}
Henceforth let $L_w := \sbm{0 & \rho^{-1} \\ \rho c(s)^2 W & \alpha c(s)^2 D}:\Zscr_w \to \Xscr_w$, and
\begin{align*}
  & \Zscr_w := \left (H_{\{1\}}\sp{1}(0,1) \cap H\sp{2}(0,1)\right ) \times
  H_{\{1\}}\sp{1}(0,1), \quad
  \Xscr_w := H_{\{1\}}\sp{1}(0,1) \times L\sp{2}(0,1) \\
  & \text{ where } \quad 
  H_{\{1\}}\sp{1}(0,1) := \left \{ f \in H\sp{1}(0,1): f(1) = 0 \right \}.
\end{align*}
The Hilbert spaces $\Zscr_w$ and $\Xscr_w$ are equipped with the norms
\begin{align*}
  & \norm{\sbm{z_1 \\ z_2}}_{\Zscr_w}^2 := \norm{z_1}_{H^2(0,1)}^2 + \norm{z_2}_{H^1(0,1)}^2 
  \quad \text{ and } \\
  & \norm{\sbm{z_1 \\ z_2}}_{H\sp{1}(0,1) \times L\sp{2}(0,1)}^2 :=
  \norm{z_1}_{H\sp{1}(0,1)}^2 + \norm{z_2}_{L\sp{2}(0,1)}^2,
\end{align*}
respectively. 
For any $\rho > 0$, the \emph{energy norm}
  \begin{equation} \label{WebstersEnergyNorm}
    \| \sbm{z_1 \\ z_2} \|_{\Xscr_w}^2 :=
    \frac{1}{2} \left (  \rho  \int_{0}^1{ \abs{z_1'(s)}^2A(s) \, ds} + 
     \frac{1}{\rho c^2} \int_{0}^1{ \abs{z_2(s)}^2 A(s) \Sigma(s)^{-2} \, ds} \right )
  \end{equation} 
  is an equivalent norm for $\Xscr_w$ because $\sqrt{2}
  \norm{z_1}_{L^2(0,1)} \leq \norm{z_1'}_{L^2(0,1)}$ for all $z_1 \in
  H_{\{1\}}\sp{1}(0,1)$.
\footnote{We denote the (strong) derivative of a (possibly
  vector-valued) function of one variable by prime. In particular,
  $f'$ denotes the $t$-derivative of load function $f = f(t,s)$ since
  it is regarded as the $L^2(0,1)$-valued function $t \mapsto
  f(t,\cdot)$. In PDE's, we denote the partial (distribution)
  derivatives by subindeces such as $\phi_{tt}$, $\phi_{ss}$, and so
  on.}  We define $\Yscr_w := \C$ with the absolute value norm $
\norm{u_0}_{\Yscr_w} := \abs{u_0}$, and the endpoint control and
observation functionals $G_w:\Zscr_w \to \Yscr_w$ and $K_w:\Zscr_w \to
\Yscr_w$ are defined by
\begin{align*}
  & G_w \sbm{z_1 \\ z_2} 
  := \frac{1}{2}\sqrt{\frac{A(0)}{ \rho  c(0)}} \left (- \rho c(0)z_1'(0) + z_2(0) \right ) \quad \text{ and } \\
  & K_w \sbm{z_1 \\ z_2} 
  := \frac{1}{2} \sqrt{\frac{A(0)}{\rho  c(0)}} \left (- \rho c(0)z_1'(0) - z_2(0) \right ).
\end{align*}
Now, the generalised Webster's equation \eqref{WebstersEqBnrCtrl}
for the state variable $x(t) = \sbm{\psi(t) \\ \pi(t)}$ can be cast in
the form
\begin{align} \label{WebstersDiffEq}
\begin{cases}
    & x'(t)  = L_w x(t) + \sbm{0 \\ \rho f(t,\cdot)}, \\
    & \tilde u(t)  = G_w x(t) ,  \quad  \tilde y(t) = K_w x(t)
   \quad \text{ for } t \in \rplus, \quad \text{ and } \\
   & x(0)  = \sbm{\psi_0 \\ \pi_0}.
\end{cases}
 \end{align}
As shown in \cite[Theorem~4.1]{A-L-M:AWGIDDS}, the triple 
\begin{equation} \label{BoundaryNode}
  \Xi^{(W)} := (G_w, L_w, K_w)
\end{equation}
 is a scattering passive, strong boundary node\footnote{It is shown in
   \cite[Theorems~4.1 and 5.1]{A-L-M:AWGIDDS} that the wave equation
   model in \eqref{IntroWaveEq} as well as the corresponding Webster's
   model in \eqref{IntroWebstersEq} are dynamical systems that can be
   represented as internally well-posed, passive boundary nodes.} on
 Hilbert spaces $(\Yscr_w,\Xscr_w,\Yscr_w)$ which is conservative if
 and only if $\alpha = 0$. For $\tilde u \in C^2(\rpluscl;\Yscr_w)$ and
 $\sbm{\psi_0 \\ \pi_0} \in \Zscr_w$, the unique classical solution of
 \eqref{WebstersDiffEq} follows in the special case that the load
 function $f$ identically vanishes (referring to the left panel in
 Fig.~\ref{FigBoth}.).

\subsection{On the weak solution of Webster's equation}

Assume that $\phi$ is a solution of the wave equation system
\eqref{IntroWaveEq} satisfying the regularity properties listed in
\eqref{WaveEqSolutionRegular} as discussed in Section~\ref{IntroSec}.
It has been shown in \cite[Theorem~3.1]{L-M:WECAD} that the averaged
solution $\bar \phi = \bar \phi(t,s)$ in \eqref{AveragedWaveEqSol}
satisfies
\begin{equation}     \label{AveragedSolutionRegular}
  \begin{aligned}
        & \bar \phi \in C^2(\rpluscl; L^2(0,1)) \quad \text{ and }
    \quad \bar \phi_s \in C^1(\rpluscl;L^2(0,1)),
  \end{aligned}
\end{equation}
and it is a \emph{weak} solution of the inhomogenous Webster's
equation
\begin{equation}  \label{WebsterInHomogenous}
     \bar \phi_{tt} - \frac{c(s)^{2}}{A(s)} \frac{\partial}{\partial s}
    \left ( A(s) \frac{\partial \bar \phi }{\partial s} \right ) 
    + \frac{2 \pi \alpha W(s) c(s)^{2} }{A(s)} \bar \phi_t
= F + G + H 
\end{equation}
where the additional load term $F + G + H \in C(\rpluscl;L^2(0,1))$ is
given by \eqref{ControlTermsF}---\eqref{ControlTermsH} above.  
This means plainly that
\begin{equation}
  \label{WebsterWeakForm}
  \begin{aligned}
    &\int_0^T\int_0^1\left(-\bar\phi_s \zeta_s-\frac{1}{c^2\Sigma(s)} \bar\phi_{tt}\zeta\right)A(s) \, dsdt 
    - 2\pi\alpha \int_0^T\int_0^1 W(s) \bar\phi_t \zeta  dsdt\\ 
    &=\int_0^T\int_0^1(F(s,t)+G(s,t)+H(s,t))\zeta(s,t)  A(s) \, dsdt 
  \end{aligned}
\end{equation}
for all test functions $\zeta\in C_0^\infty((0,1)\times (0,T))$ and
all $T>0$.  

Now, fix $t_0 \in (0,T)$ and let $\{ v_\epsilon \} \subset
C_0^\infty(0,T)$ for $\epsilon > 0$ be a family of non-negative
functions such that $\int_0^T{v_\epsilon \, dt} = 1$ and $\lim_{\epsilon
  \to 0}{v_\epsilon(t)} = 0$ for all $t \in (0,T)\setminus\{t_0 \}$.
Let $\xi \in C_0^\infty(0,1)$ and define $\zeta(s,t) := \xi(s)
v_{\epsilon}(t)$. By Fubini's Theorem, we get from \eqref{WebsterWeakForm}
\begin{equation} \label{WebsterWeakForm2}
  \begin{aligned}
    &\int_0^T \left ( \int_0^1\left(-\bar\phi_s(t,s)
    \xi_s(s)-\frac{1}{c^2\Sigma(s)} \bar\phi_{tt}(t,s)\xi(s)  \right)A(s) \, ds
    \right )v_{\epsilon}(t)\, dt \\
    & - 2\pi\alpha \int_0^T \left ( \int_0^1 W(s) \bar\phi_t(t,s) \xi(s)  ds \right ) v_{\epsilon}(t) \, dt
    \\ &=\int_0^T \left ( \int_0^1(F(s,t)+G(s,t)+H(s,t))\xi(s) A(s) \,  ds \right ) v_{\epsilon}(t) \,dt
  \end{aligned}
\end{equation}
By \eqref{AveragedSolutionRegular} and the fact that $F + G + H \in
C(\rpluscl;L^2(0,1))$, the three inner integrals in
\eqref{WebsterWeakForm2} represent continous functions in variable
$t$.  By letting $\epsilon \to 0$, we get the identity
\begin{equation*} 
  \begin{aligned}
    & \int_0^1\left(-\bar\phi_s(t_0,s)
    \xi_s(s)-\frac{1}{c^2\Sigma(s)} \bar\phi_{tt}(t_0,s)\xi(s)  \right)A(s) \, ds \\
    & - 2\pi\alpha  \int_0^1 W(s) \bar\phi_t(t_0,s) \xi(s)  ds
    \\ &= \int_0^1(F(s,t_0)+G(s,t_0)+H(s,t_0))\xi(s) A(s) \,  ds.
  \end{aligned}
\end{equation*}
This means that \eqref{WebsterInHomogenous} holds pointwise for all $t
= t_0 > 0$ if the four terms in \eqref{WebsterInHomogenous} are
regarded as \emph{distributions} for each fixed $t \in (0,1)$.  By
\eqref{AveragedSolutionRegular} and $F + G + H \in
C(\rpluscl;L^2(0,1))$, all other terms except the second in
\eqref{WebsterInHomogenous} are functions in $L^2(0,1)$ for any fixed
$t \in (0,1)$. We conclude that the equality in
\eqref{WebsterInHomogenous} holds in $L^2(0,1)$ (understood as a
subspace of distributions) for each fixed $t > 0$. Even the second
term in \eqref{WebsterInHomogenous} satisfies
\begin{equation} \label{SecondOrderTermInL2}
\frac{c(s)^{2}}{A(s)} \frac{\partial}{\partial s} \left ( A(s)
\frac{\partial \bar \phi }{\partial s} \right ) \in
C(\rpluscl;L^2(0,1)).
\end{equation}
By continuity, Webster's equation \eqref{WebsterInHomogenous} holds
with equality in $C(\rpluscl;L^2(0,1))$. This is the reformulation of
\cite[Theorem~3.1]{L-M:WECAD} that we use in this article.

\begin{lemma} \label{TheAverageIsSolutionInSolutionSpace}
  Let the functions $\phi$, $\bar \phi$, $F$, $H$, and $H$ be defined
  as above. Then $x(t) = \sbm{\bar \phi(t, \cdot) \\ \rho \bar
    \phi_t(t, \cdot)}$ is a solution of the first equation in
  \eqref{WebstersDiffEq} where $f = F + G + H$ and $L_w$ is given in
  Section~\ref{InHomoWeb}.
\end{lemma}
\begin{proof}
  We first show that $x(t) \in \Zscr_w = \Dom{L}$ for all $t \geq
  0$. By the latter inclusion in \eqref{AveragedSolutionRegular} and
  the fact that $\bar \phi(t,1) = 0$ for all $t \geq 0$, we get $\bar
  \phi(t,\cdot) \in H^1_{\{1\}}(0,1)$. Because $A(\cdot)$ is continuously
  differentiable, it follows from \eqref{SecondOrderTermInL2} that
  $\bar \phi_{ss}(t,\cdot) \in L^2(0,1)$; implying $\bar \phi(t,\cdot)
  \in H^2(0,1)$.

  By the latter inclusion in \eqref{AveragedSolutionRegular}, $\bar
  \phi \in C^1(\rpluscl;H^1(0,1))$. Hence, $\bar \phi_t(t, \cdot) \in
  H^1_{1}(0,1)$ since $\bar \phi_t(t,1) = 0$ as a consequence of $\bar
  \phi(t,1) = 0$. We conclude that $\bar \phi_t(t,\cdot) \in
  H^1_{\{1\}}(0,1)$. We have now shown that $x(t) \in \Zscr_w$ for all
  $t$.

  The claim follows from
\begin{equation*}
  L_w \bbm{\bar \phi(t, \cdot) \\ \rho \bar \phi_t(t, \cdot)} =
  \bbm{\bar \phi_t(t, \cdot) \\ \rho c(s)^2 \left ( W \bar \phi(t,
    \cdot) + \alpha D \bar \phi_t(t, \cdot) \right )} = \bbm{\bar
    \phi_t(t, \cdot) \\ \rho \left (\bar \phi_{tt}(t, \cdot) - f(t,
    \cdot) \right )} 
\end{equation*}
where the last equality is by \eqref{WebsterInHomogenous}.  In
particular, $L_w x \in \Xscr_w$.
\end{proof}

As a consequence of \eqref{SecondOrderTermInL2} and \eqref{StandingSmoothness},
the averaged solution $\bar \phi$ as a little more regularity:
\begin{lemma} \label{ALittleBitMoreRegularityLemma}
  The function $\bar \phi$ defined above satisfies $\bar \phi \in
  H^2(0,1)$.
\end{lemma}

\subsection{On system nodes}

To treat the case $f \neq 0$ in \eqref{WebstersDiffEq}, we rewrite
\eqref{WebstersDiffEq} in terms of \emph{system nodes} in
Section~\ref{WebsterDirectSec}. There exists a wide literature on
system nodes, and we give a short reminder on what we need based on
\cite{M-S-W:HTCCS,OS:WPLS}.

Following \cite[Definition 2.1]{M-S:CBCS} or \cite[Definition
  2.2]{M-S-W:HTCCS}, the system node is characterised as follows:
\begin{definition} \label{SysNodeDef}
An operator
$$S := \SysNode \colon \Xscr \times \Uscr \supset \Dom
S \to \Xscr \times \Yscr$$ is called an \emph{system node} on the
Hilbert spaces $(\Uscr,\Xscr,\Yscr)$ if the following holds:
\begin{enumerate}

\item 
$A$ is a generator of a strongly continuous semigroup on $\Xscr$.

\item $B \in \BLO(\Uscr;\Xscr_{-1})$ where $\Xscr_{-1} = \Dom{A^*}^d
  \supset \Xscr$ is the usual extrapolation space.

\item $\Dom S = \{ \sbm{x \\u }
\in \Xscr \times \Uscr :  A_{-1}x + Bu \in \Xscr \}$
where  $A_{-1} \in
\BLO(\Xscr;\Xscr_{-1})$ is the Yoshida extension of $A$.

\item $\AB=\bbm{A_{-1} & B}\rst{\Dom S}$. 

\item $\CD \in \BLO(\Dom S;\Yscr)$ where
we use on $\Dom S$ the graph norm of
$\AB$:
$$ \left \| \sbm{x \\ u} \right \|_{\Dom S}^2 := \norm x_{\Xscr}^2 +
   \norm u_{\Uscr}^2 + \norm{A_{-1}x + Bu}_{\Xscr}^2 .$$
\end{enumerate}
\end{definition}
\noindent Details of $A_{-1}$ and $\Xscr_{-1}$ can be found in, e.g.,
\cite[Proposition 2.1]{M-S-W:HTCCS}. We also use the Hilbert space
$\Xscr_{1} = \Dom{A}$ equipped with the graph norm of $A$. Whenever we
refer to these spaces for the \emph{dual node} $S^d$ (as characterised
in \cite[Proposition 2.4]{M-S-W:HTCCS}), we use the symbols
$\Xscr_{1}^d$ and $\Xscr_{-1}^d$.


The dynamical equations for systems nodes  take the
form that is reminiscent of the equations in finite-dimensional linear
system theory where $S = \sbm{A & B \\ C & D}$ :
\begin{equation}\label{SSSS}
 \bbm{x'(t) \\ \tilde y(t)} = S \bbm{x(t) \\ \tilde u(t)} 
\text{ for } t \in \rplus;  \quad 
 \quad x(0) = x_0. 
 \end{equation}
\begin{proposition} \label{SystemNodeSolvabilityProp}
Assume that $S = \SmallSysNode$ is a system node with domain
$\Dom{S}$.  For all $x_0 \in \Xscr$ and $\tilde u \in
C^2(\rpluscl;\Uscr)$ with $\sbm{x_0 \\ \tilde u(0)} \in \Dom{S}$ the
equations \eqref{SSSS} are uniquely solvable, and the solutions
satisfy $x \in C^1(\rpluscl;\Xscr)$, $\tilde y \in C(\rpluscl;\Yscr)$, and
$\sbm{x \\ \tilde u} \in C(\rpluscl; \Dom{S})$.
\end{proposition}
\noindent This is given in \cite[Proposition 2.6]{M-S-W:HTCCS}, and
these solutions are called \emph{classical} in the sense of
mathematical systems theory. For a more complete treatment of system
nodes, see \cite[Section 2]{M-S-W:HTCCS}.

\section{\label{WebsterDirectSec} Inhomogenous Webster's model}

The purpose of this section is to rewrite the inhomogenous Webster's
model \eqref{WebstersDiffEq} as a system node with an energy
inequality.

As argued in \cite[Section 2]{M-S:CBCS}, boundary node $\Xi^{(W)} =
(G_w, L_w, K_w)$ from \eqref{BoundaryNode} induces a unique system
node $S = \SmallSysNode$ on Hilbert spaces $(\Yscr_w, \Xscr_w,
\Yscr_w)$ with operators $A$, $A_{-1}$, $B$, and $\CD$ as in
Definition~\ref{SysNodeDef}.  Then, if $\tilde u \in
C^2(\rpluscl;\Yscr_w)$ and $x_0 = \sbm{\psi_0 \\ \pi_0} \in \Zscr_w$,
the functions $x$, $\tilde y$ in \eqref{WebstersDiffEq} and
\eqref{SSSS} are the same if $f \equiv 0$ in \eqref{WebstersDiffEq}.
The node $S$ is of boundary control form in the sense that $B \Yscr_w
\cap \Xscr_w = \{ 0 \}$ and $\Null{B} = \{ 0 \}$, and we make use of
the following relations\footnote{A shorter way of writing all this is
  $\sbm{L_w \\ K_w}= S \sbm{I \\ G_w}$.} connecting $S$ and
$\Xi^{(W)}$: $\Dom{S} = \sbm{I \\ G_w} \Zscr_w$, $A =
L_w\rst{\Null{G_w}}$ with $\Dom{A}= \Null{G_w}$, $L_w =
A_{-1}\rst{\Zscr_w} + B G_w$, and $\CD = \bbm{K_w & 0}\rst{\Dom{S}}$;
for details, see, e.g., \cite[Section~2.2]{M-S:CBCS}. The unbounded
adjoint of $A$ satisfies $A^* := -L_w \rst{\Null{K_w}}$ with
$\Dom{A^*} = \Null{K_w}$ by \cite[Theorem 1.7 and Proposition
  4.3]{M-S:CBCS}. To write \eqref{WebstersDiffEq} as a system node,
say $S^{(W)}$, amounts to augmenting $S$ with an additional input that
accommodates the load term $f$.

We define the Hilbert spaces $(\Xscr_w)_1 := \Dom{A}$ and
$(\Xscr_w)_1^* := \Dom{A^*}$ with the graph norms
$\norm{z}_{(\Xscr_w)_1}^2 = \norm{A z}_{\Xscr_w}^2 + \norm
z_{\Xscr_w}^2$ and $\norm{z}_{(\Xscr_w)_1^*}^2 = \norm{A^*
  z}_{\Xscr_w}^2 + \norm z_{\Xscr_w}^2$, respectively. Define
$(\Xscr_w)_{-1}$ to be the dual of $\Dom{A_w^*}$ when we identify the
dual of $\Xscr_w$ with itself.  Then $(\Xscr_w)_1 \subset \Xscr_w
\subset (\Xscr_w)_{-1}$ with continuous and dense
embeddings.\footnote{Recall that $(\Xscr_w)_1 \subset \Zscr_w \subset
  \Xscr_w$ but $(\Xscr_w)_1$ is not dense in $\Zscr_w$.}  With these
definitions, $B \in \BLO(\Yscr_w ;(\Xscr_w)_{-1})$.

Define the control operators $B^{(e)} := \sbm{0 \\ \rho}:L^2(0,1) \to
\Xscr_w$ and $B_w := \bbm{B & B^{(e)}} \in \BLO(\Uscr_w;
(\Xscr_w)_{-1})$ where $\Uscr_w := \Yscr_w \times L^2(0,1)$ with the
norm \linebreak $\norm{\sbm{\tilde u \\ f}}_{\Uscr_w}^2 = \norm{\tilde
  u}_{\Yscr_w}^2 + \norm{f}_{L^2(0,1)}^2$. Define $\Dom{S^{(W)}} :=
\Dom{S} \times L^2(0,1)$ (where $\Dom{S} = \sbm{I \\ G_w} \Zscr_w$)
with the norm
 \begin{equation*} \label{ABGraphNorm}
 \norm{\sbm{z \\ \tilde  u \\ f}}_{\Dom{S^{(W)}}}^2
 = \norm{ z }_{\Zscr_w}^2 + \norm{ G_w z }_{\Uscr_w}^2  + \norm f_{L^2(0,1)}^2
\end{equation*}
and the operators
\begin{equation*}
  [\AB]_w := \bbm{A_{-1} & B_w}\rst{\Dom{S^{(W)}}} \quad \text{ and } \quad  [\CD]_w :=
\bbm{\CD & 0}\rst{\Dom{S^{(W)}}}
\end{equation*}
yields now the system node 
\begin{equation} \label{WebstersSystemNode}
  S^{(W)} := \bbm{[\AB]_w \cr [\CD]_w} 
\end{equation}
on the Hilbert spaces $(\Uscr_w,\Xscr_w,\Yscr_w)$ with domain
$\Dom{S^{(W)}}$. It is clear from the construction that $S^{(W)}$ has been
obtained by adding a new input (using the operator $B^{(e)}$ above) to the
system node $S$ that is associated to boundary node $\Xi^{(W)}$ by
\cite[Theorem 2.3]{M-S:CBCS}.

The node $S^{(W)}$ is, in particular, internally well-posed since it has
the same semigroup as $S$.  Hence, for any $\sbm{\psi_0 \\ \pi_0} \in \Zscr_w$ and
$\sbm{\tilde u \\ f } \in C^2(\rpluscl;\Uscr_w)$ satisfying the
compatibility condition $G_w \sbm{\psi_0 \\ \pi_0} = \tilde u(0)$, the first and the
last of the equations in
\begin{equation} \label{WebsterSystemNodeCauchy}
  \begin{cases}
    x'(t) & = A_{-1} x(t) + B_w \sbm{\tilde u(t)\\ f(t,\cdot)}, \\ \tilde  y(t) & =
     [\CD]_w \, \sbm{x(t) \\ \tilde u(t)\\ f(t,\cdot)} \quad \text{ for } \quad t
     \in \rplus, \quad \text{ and }\\ x(0) &= \sbm{\psi_0 \\ \pi_0}
\end{cases}
\end{equation}
have a unique classical solution $x \in C^1(\rpluscl;\Xscr_w)$ with
$\sbm{x \\ \tilde u \\ f} \in C(\rpluscl;\Dom{S^{(W)}})$.  (These
equations are plainly \eqref{SSSS} written for $S^{(W)}$ instead of
$S$.)  Then the output signal can be defined through the second of the
equations in \eqref{WebsterSystemNodeCauchy} since $[\CD]_w \in
\BLO(\Dom{S^{(W)}}; \Uscr_w)$ as in
Proposition~\ref{SystemNodeSolvabilityProp}.  We conclude that
\eqref{WebstersDiffEq} and \eqref{WebsterSystemNodeCauchy} are
equivalent Cauchy problems under the assumptions on $\sbm{\psi_0
  \\ \pi_0}$ and $\sbm{\tilde u \\ f }$ stated above.

The state $x(\cdot)$ in equations \eqref{WebsterSystemNodeCauchy} is
controlled both from the boundary points $0,1$ (using the control
function $\tilde u$) and also from all of the interior points of the
interval $[0,1]$ (using the control function $f$).  We show next that
that if both $\tilde u$ and $f$ are twice continuously differentiable
in time, the boundary and the interior point parts of the control ``do
not mix''.
\begin{proposition} \label{FirstStateSplittinProp}
  Let $\sbm{\psi_0 \\ \pi_0} \in \Zscr_w$, $\sbm{ \tilde u \\ f } \in
  C^2(\rpluscl;\Uscr_w)$, and $G_w \sbm{\psi_0 \\ \pi_0} = \tilde
  u(0)$.  Then the classical solution $x$ of the first and the last of
  equations \eqref{WebsterSystemNodeCauchy} (associated with the
  system node in \eqref{WebstersSystemNode}) satisfies $x = z + w$
  where $z$ is the classical solution of \eqref{WebstersDiffEq} with
  $f \equiv 0$ (associated with the boundary node $\Xi^{(W)}$ in
  \eqref{BoundaryNode}), and $w(t) \in \Null{G_w}$ for all $t \geq 0$.
\end{proposition}
\begin{proof}  
  By linearity, the classical solution $x$ of
  \eqref{WebsterSystemNodeCauchy} can be decomposed as the sum $x = z
  + w$ of two classical solutions $z$ and $w$ for $t \in \rplus$ of
  the equations
  \begin{equation} \label{BoundaryControlEquation} 
    z'(t) = A_{-1} z(t) + B_w \sbm{\tilde u(t) \\ 0 } = A_{-1}
    z(t) + B \tilde u(t) \text{ with }  z(0) = \sbm{\psi_0 \\ \pi_0};
  \end{equation} 
  and
  \begin{equation} \label{PerturbationEquation} 
    w'(t) = A_{-1} w(t) + B_w \sbm{0 \\ f(t,\cdot)}
    =  A_{-1} w(t) + B^{(e)} f(t,\cdot)  
     \text{ with } w(0) = 0.
  \end{equation} 
  Because the operators $A_{-1}$ and $B$ relate to $S$ (as introduced
  in the beginning of this section) and, hence, to the boundary node
  $\Xi^{(W)}$ in \eqref{BoundaryNode}, equations
  \eqref{BoundaryControlEquation} give $z'(t) = L z(t) + B(\tilde u(t) -
  Gz(t)) = Lz(t)$, $B(\tilde u(t) - Gz(t)) = 0$, and hence $ \tilde u(t)= G_w z(t)$
  because $\Null{B} = \{ 0 \}$.
  
  Consider next the initial value problem
  \begin{equation} \label{PerturbationEquation2} 
      \tilde w'(t)  = A_{-1} \tilde w(t) + B^{(e)} f'(t,\cdot)  
      \quad \text{ for } \quad   t \in \rplus, \quad \tilde w(0) = B^{(e)} f(0),
  \end{equation} 
  where now $f' \in C^1(\rpluscl;L^2(0,1))$ and $\tilde w(0) \in
  \Xscr_w$.  Denote by $T(\cdot)$ the strongly continuous contraction
  semigroup on $\Xscr_w$ generated by $A$.  Because $B^{(e)} \in
  \BLO(L^2(0,1);\Xscr_w)$, the variation of constants formula $\tilde
  w(t) = T(t)B^{(e)} f'(0,\cdot) + \int_{0}^t{T(t - \tau)B^{(e)}
    f'(\tau,\cdot)\, d\tau}$ gives a unique strong solution of
  \eqref{PerturbationEquation2} satisfying $\tilde w \in
  C(\rpluscl;\Xscr_w)$; see \cite[Theorem 3.8.2(iv)]{OS:WPLS}. Then $w$
  defined by $w(t) := \int_{0}^{t}{\tilde w(\tau)\, d\tau}$ satisfies
  $\tilde w(t) = A_{-1} w(t) + B^{(e)} f(t,\cdot)$ for all $t \geq 0$, as
  can be seen by integrating \eqref{PerturbationEquation2} over
  $[0,1]$ as a $(\Xscr_w)_{-1}$-valued function. Since also $\tilde w
  = w'$ (derivative computed in the space $(\Xscr_w)_{-1}$), we
  conclude that $w$ equals the unique classical solution of
  \eqref{PerturbationEquation}, with $w \in C^1(\rpluscl;\Xscr_w)$.
  
  It now follows from $A_{-1} w(t) = \tilde w(t) - B^{(e)} f(t,\cdot)$
  that $w \in C(\rpluscl;(\Xscr_w)_1)$. Therefore $G_w w(t) = 0$
  because $(X_w)_1 = \Dom{A} = \Null{G_w}$.
\end{proof}
In fact, the system node $S^{(W)}$ defines a well-posed linear system
in the usual sense of, e.g., \cite[Definition 2.7]{M-S-W:HTCCS} and
\cite[Definition 2.2.1]{OS:WPLS}:
\begin{theorem}
  \label{WebsterNodeIsWellPosedThm}
  The classical solution of \eqref{WebsterSystemNodeCauchy} satisfies
  the energy inequality
  \begin{equation}
    \label{WebsterWellPosedInEq}
    \frac{d}{dt}\norm{x(t)}_{\Xscr_w}^2 \leq \abs{ \tilde u(t)}^2 + 2 \rho \cdot \Re
    \left <x(t), \sbm{0 \\ f (t,\cdot)} \right >_{\Xscr_w} -
    \abs{\tilde y(t)}^2
  \end{equation}
  for all $t > 0$. Moreover,  the well-posedness estimate
  \begin{equation}
    \label{WebsterWellPosedInEq2}
    \norm{x(T)}_{\Xscr_w}^2 + \norm{\tilde y}_{L^2((0,T);\Yscr_w)}^2 \leq K(T)
    \left (\norm{\sbm{\psi_0 \\ \pi_0}}_{\Xscr_w}^2 + \norm{\sbm{ \tilde u \\ f}}_{L^2((0,T);\Uscr_w)}^2\right )
  \end{equation}
  holds for all $T \geq 0$ where $K(T) := 5 (\rho + 1)^{1/2} (T + 1)$.
\end{theorem}
\begin{proof}
  We first verify \eqref{WebsterWellPosedInEq} for the classical
  solution $x$ of \eqref{WebsterSystemNodeCauchy} for which
  $\sbm{\psi_0 \\ \pi_0} \in \Zscr_w$, $\sbm{ \tilde u \\ f } \in
  C^2(\rpluscl;\Uscr_w)$, and $G_w \sbm{\psi_0 \\ \pi_0} = \tilde
  u(0)$. Proposition~\ref{FirstStateSplittinProp} gives the
  decomposition $x(t) = z(t) + w(t) \in \Zscr_w$ for such solutions
  where $z'(t) = L_wz(t)$, $w(t) \in \Null{G_w}$, and $w'(t) = Aw(t) +
  B^{(e)} f(t,\cdot)$, we get for any $t \geq 0$
  \begin{equation}
    \label{WebsterNodeIsWellPosedPropEq1a}
    \begin{aligned}
      & \frac{d}{dt}\norm{x(t)}_{\Xscr_w}^2 + \abs{\tilde y(t)}^2
      = 2 \Re \left <x(t), z'(t) + w'(t) \right >_{\Xscr_w} + \abs{\tilde y(t)}^2 \\
      & = 2 \Re \left <x(t), L_w x(t) \right >_{\Xscr_w} + \abs{\tilde y(t)}^2  
      + 2 \Re \left <x(t), B^{(e)} f(t,\cdot) \right >_{\Xscr_w}
    \end{aligned}
\end{equation}
  since $A_{-1}\rst{\Zscr_w } = L_w - B G_w$ and $A =
  L_w\rst{\Null{G_w}}$.  Since 
  \begin{equation*}
    \tilde y(t) = \left [\CD \right ]_w \left ( \sbm{z(t)
      \\ \tilde u(t) \\ 0} + \sbm{w(t) \\ 0 \\ f(t,\cdot)} \right ) = K_w (z(t) + w(t)) =
    K_w x(t),
  \end{equation*}
  we have by the passivity of $\Xi^{(W)}$ the
  Green--Lagrange inequality
  \begin{equation*}
    2 \Re \left <x(t), L_w x(t) \right >_{\Xscr_w} + \abs{K_w x(t)}^2
    \leq \abs{G_w x(t)}^2 = \abs{\tilde u(t)}^2.
  \end{equation*}
  This, together with \eqref{WebsterNodeIsWellPosedPropEq1a}, gives
  for all $t \geq 0$ the energy estimate
  \begin{equation}
    \label{WebsterNodeIsWellPosedPropEq1}
    \frac{d}{dt}\norm{x(t)}_{\Xscr_w}^2 + \abs{\tilde y(t)}^2
    \leq  \abs{ \tilde u(t)}^2 + 2 \Re \left <x(t), B^{(e)} f(t,\cdot) \right >_{\Xscr_w}. 
  \end{equation}
  Since $B^{(e)} f (t,\cdot) = \sbm{0 \\ \rho f (t,\cdot)}$, we
  conclude that \eqref{WebsterWellPosedInEq} holds.

  To conclude \eqref{WebsterWellPosedInEq2} from
  \eqref{WebsterWellPosedInEq}, we must obtain an \emph{a priori}
  bound for $\norm{x(t)}_{\Xscr_w}$. We use again the splitting $x = z
  + w$ from Proposition~\ref{FirstStateSplittinProp}. Because
  \eqref{BoundaryControlEquation} describes the input part of the
  scattering passive system node $S$ associated to $\Xi^{(W)}$ in
  \eqref{BoundaryNode}, we get
  \begin{equation} \label{WebsterNodeIsWellPosedPropEq2}
    \norm{z(t)}_{\Xscr_w}^2 \leq \norm{\sbm{\psi_0 \\ \pi_0}}_{\Xscr_w}^2 
    + \norm{\tilde u}_{L^2(0,t)}^2
  \end{equation}  
  As in the proof of Proposition~\ref{FirstStateSplittinProp}, the
  variation of constants formula gives $w(t) = \int_{0}^t{T(t - \tau)
    \sbm{0 \\ \rho f(\tau,\cdot)}\, d\tau}$ for the solution of
  \eqref{PerturbationEquation2}.  Because $T(\cdot)$ is a contraction
  semigroup, it follows from H\"older's inequality that
  \begin{equation}\label{WebsterNodeIsWellPosedPropEq3} 
    \norm{w(t)}_{\Xscr_w}^2 \leq t \rho^2 \norm{f}_{L^2((0,t);L^2(0,1))}^2 
    = t \rho^2 \norm{f}_{L^2((0,1)\times(0,t))}^2.
  \end{equation}  
  Combining \eqref{WebsterNodeIsWellPosedPropEq2} and
  \eqref{WebsterNodeIsWellPosedPropEq3} we get 
  \begin{equation*}
    \norm{x(t)}_{\Xscr_w}^2 
 \leq  2 (\norm{z(t)}_{\Xscr_w}^2   +     \norm{w(t)}_{\Xscr_w}^2   ) 
\leq 2 ( \norm{\sbm{\psi_0 \\ \pi_0}}_{\Xscr_w}^2 
+  (1 + t\rho^2) \norm{\sbm{\tilde u \\ f}}_{L^2((0,t);\Uscr_w)}^2)
  \end{equation*}
  and thus $\norm{x}_{L^2((0,T);\Xscr_w)}^2 \leq 2 T \norm{\sbm{\psi_0
      \\ \pi_0}}_{\Xscr_w}^2 + (\rho T^2 + 2T ) \norm{\sbm{\tilde u
      \\ f}}_{L^2((0,T);\Uscr_w)}^2) \leq (\rho T^2 + 2T ) (
  \norm{\sbm{\psi_0 \\ \pi_0}}_{\Xscr_w}^2 + \norm{\sbm{ \tilde u
      \\ f}}_{L^2((0,T);\Uscr_w)}^2)$ which implies
\begin{equation*}
 \norm{x}_{L^2((0,T);\Xscr_w)} \leq (\rho + 1)^{1/2} (T + 1) (
 \norm{\sbm{\psi_0 \\ \pi_0}}_{\Xscr_w} + \norm{\sbm{\tilde u
     \\ f}}_{L^2((0,T);\Uscr_w)}).
\end{equation*}
Now we get
\begin{align*}
  & \int_0^T{\abs{\left <x(t), B^{(e)} f (t,\cdot) \right >_{\Xscr_w}} \, dt} 
  \leq \norm{x}_{L^2((0,T);\Xscr_w)} \cdot \norm{\sbm{0 \\ \rho f()}}_{L^2((0,T);\Xscr_w)} \\ 
  & \leq (\rho + 1)^{3/2} (T + 1) (\norm{\sbm{\psi_0 \\ \pi_0}}_{\Xscr_w} + \norm{\sbm{\tilde u
      \\ f}}_{L^2((0,T);\Uscr_w)}) \cdot \norm{f}_{L^2((0,1)\times (0,T))} \\ 
  & \leq (\rho + 1)^{3/2} (T + 1) (\norm{\sbm{\psi_0 \\ \pi_0}}_{\Xscr_w} 
  + \norm{\sbm{\tilde u \\ f}}_{L^2((0,T);\Uscr_w)})^2 \\ 
  & \leq 2 (\rho + 1)^{3/2} (T + 1) (\norm{\sbm{\psi_0 \\ \pi_0}}_{\Xscr_w}^2 + \norm{\sbm{\tilde u
      \\ f}}_{L^2((0,T);\Uscr_w)}^2 ).
\end{align*}
This together with \eqref{WebsterWellPosedInEq} produces
\eqref{WebsterWellPosedInEq2} provided that $\sbm{\psi_0 \\ \pi_0} \in \Zscr_w$, $\sbm{\tilde u
  \\ f } \in C^2(\rpluscl;\Uscr_w)$, and $G_w \sbm{\psi_0 \\ \pi_0} =  \tilde u(0)$.
\end{proof}

Using the well-posedness estimate of
Theorem~\ref{WebsterNodeIsWellPosedThm}, we can move from classical
solutions to more general \emph{strong solutions} of equations
\eqref{WebsterSystemNodeCauchy}.
\begin{corollary}   \label{WebsterNodeIsWellPosedCor}
  The system node $S^{(W)}$ in \eqref{WebstersSystemNode}, associated
  to the inhomogenous Webster's equation described by
  \eqref{WebstersEqBnrCtrl}---\eqref{ControlTermsH}, defines a well-posed
  linear system through equations \eqref{WebsterSystemNodeCauchy}.

  The first and the last of equations in
  \eqref{WebsterSystemNodeCauchy} have a unique \emph{strong} solution
  $x$ (in $\Xscr_w$) for any $\sbm{\psi_0 \\ \pi_0} \in \Xscr_w$ and
  $\sbm{ \tilde u\\ f} \in L^2(\rplus;\Uscr_w)$ satisfying $x \in
  C(\rpluscl;\Xscr_w) \cap W_{loc}^{1,1}(\rplus;(\Xscr_w)_{-1})$.  The
  output function satisfies $\tilde y \in L^2_{loc}(\rplus;\Yscr_w)$,
  and the well-posedness estimate \eqref{WebsterWellPosedInEq2} holds.
\end{corollary}
\noindent 
 Strong solutions are defined in \cite[Definition 3.8.1]{OS:WPLS} in
 the sense of mathematical systems theory. It is clear that classical
 solutions of \eqref{WebsterSystemNodeCauchy} (as given in
 Proposition~\ref{SystemNodeSolvabilityProp}) are strong solutions as
 well. Conversely, it does not make sense to say that a strong
 solution would in general satisfy equations in \eqref{WebstersDiffEq}
 for, e.g, $\sbm{ \tilde u\\ f} \notin C^2(\rpluscl;\Uscr_w)$.
\begin{proof}
That $S^{(W)}$ defines a well-posed linear system follows from
estimate \eqref{WebsterWellPosedInEq2} and \cite[Lemma 4.7.8 and
  Theorem 4.7.15]{OS:WPLS}. The existence of the strong solution
follows from the definition of the well-posed linear system; see
\cite[Definition~2.2.1]{OS:WPLS}. That the strong solution satisfies
\eqref{WebsterWellPosedInEq2} follows by density as given in
\cite[Definition~2.7]{M-S-W:HTCCS} and the discussion following it.
\end{proof}

\section{\label{ErrorDynamicsSec}Tracking error dynamics}

It is now time to discuss in a rigorous way what actually is described
in the right panel of Fig.~\ref{FigBoth}. There, both the wave
equation and Webster's equation are boundary controlled by a common
external signal, apart from averaging. More precisely, the boundary
control signal $u \in C^2(\rpluscl;L^2(\Gamma(0)))$ acts as an input
for the wave equation, and the scalar signal
\begin{equation} \label{AveragedInputFunction}
 \bar u(t) : = \frac{1}{A(0)} {\int_{ \Gamma(0) }{u d A}} \quad \text{ for } \quad  t \in \rpluscl 
\end{equation}
satisfying $\bar u \in C^2(\rpluscl;\Yscr_w)$ is used as the input for
the Webster's model. It has been shown in
\cite[Theorem~3.1]{L-M:WECAD} that the averaged solution $\bar \phi =
\bar \phi(t,s)$ in \eqref{AveragedWaveEqSol}, with $\phi$ coming from
\eqref{IntroWaveEq}, is a \emph{weak} solution $\psi = \bar \phi$ of
the problem
\begin{equation} \label{WebstersEqBnrInteriorCtrl}
  \begin{cases}
    & \psi_{tt} - \frac{c(s)^{2}}{A(s)} \frac{\partial}{\partial s}
    \left ( A(s) \frac{\partial \psi }{\partial s} \right ) 
    + \frac{2 \pi \alpha W(s) c(s)^{2} }{A(s)} \psi_t
= f(t,s)
\\   & \hfill \text{ for }   s \in (0,1) \text{ and } t \in \rplus, \\
    & - c \psi_s(t,0) + \psi_t(t,0) = 2 \sqrt{\frac{c}{\rho
        A(0)}} \, \bar u(t) \quad
    \text{ for } t \in \rplus,  \\
    & \psi(t,1) = 0 \quad  \text{ for } t \in \rplus, \quad \text{ and } \\
    & \psi(0,s) = \bar \phi(0,s), \quad \psi_t(0,s) = \bar \phi_t(0,s) \quad
    \text{ for } s \in (0,1),
\end{cases}
\end{equation}
where the additional load term $f = F + G + H \in
C(\rpluscl;L^2(0,1))$ is given by
\eqref{ControlTermsF}---\eqref{ControlTermsH} above.  By
\cite[Theorem~3.1]{L-M:WECAD}, the particular weak solution $\bar
\phi$ of \eqref{WebstersEqBnrInteriorCtrl} has extra regularity a
consequence of \eqref{WaveEqSolutionRegular}, namely
\begin{equation}     \label{AveragedSolutionRegular}
  \begin{aligned}
        & \bar \phi \in C^2(\rpluscl; L^2(0,1)) \quad \text{ and }
    \quad \bar \phi_s \in C^1(\rpluscl;L^2(0,1)).
  \end{aligned}
\end{equation}

On the other hand, the system described by
\eqref{WebstersEqBnrInteriorCtrl} and the output function $\tilde y$
defined by \eqref{WebstersEqBnrCtrlObs} can be reformulated in terms
of the scattering passive system node as
\begin{equation}\label{WebsterNodeDynamics}
 \bbm{x'(t) \\ \tilde y(t)} = S^{(W)} \sbm{x(t) \\ \bar
   u(t)\\ f(t,\cdot)} \quad \text{ and } \quad \quad x(0) = \sbm{\bar
   \phi(0,\cdot) \\ \rho \bar \phi_t(0,\cdot)}
 \end{equation}
as shown in Section~\ref{WebsterDirectSec}.   Equation
\eqref{WebsterNodeDynamics} has a \emph{unique strong solution} $x$ by
Corollary~\ref{WebsterNodeIsWellPosedCor} which is of the form $x =
\sbm{\psi \\ \pi}$ where $\psi$ solves
\eqref{WebstersEqBnrInteriorCtrl}, and $\pi = \rho \psi_t$.  
To apply the estimate \eqref{WebsterWellPosedInEq2} using
Corollary~\ref{WebsterNodeIsWellPosedCor}, we need to conclude that
the top component $\psi$ of the strong solution $x$ of
\eqref{WebsterNodeDynamics} equals $\bar \phi$ for all $t \geq 0$.
\begin{lemma} \label{LiftingWeakSolutionToStrongLemma}
Let $\Omega$ an $\Gamma(0) \subset \partial \Omega$ be defined as in
Section~\ref{IntroSec}, and let $u \in C^2(\rpluscl; L^2(\Gamma(0)))$.
By $\phi$ denote the solution of the wave equation model
\eqref{IntroWaveEq} safisfying the regularity conditions
\eqref{WaveEqSolutionRegular}, and define $y \in C(\rpluscl;
L^2(\Gamma(0)))$ by \eqref{IntroWaveObs}.  Assume that
\begin{enumerate}
 \item the function $\bar \phi$ is obtained from $\phi$ of
   \eqref{IntroWaveEq} by the averaging operator given in
   \eqref{AveragedWaveEqSol};
 \item the function $\bar u \in C^2(\rpluscl;\Yscr_w)$ is obtained
   from $u$ by \eqref{AveragedInputFunction}; the function $\bar y \in
   C^1(\rpluscl;\Yscr_w)$ is obtained similarly from $y$; and
\item the function $f \in C(\rpluscl;L^2(0,1))$ is defined as $f = F +
  G + H$ where $F$,$G$, and $H$ are given by
  \eqref{ControlTermsF}--\eqref{ControlTermsH}.
\end{enumerate}
Then $x(t) = \sbm{\bar \phi(t,\cdot) \\ \rho \bar \phi_t(t,\cdot)}$ is
the (unique) strong solution of \eqref{WebsterNodeDynamics} with the
output satisfying $\tilde y = \bar y$.
\end{lemma}
\noindent We remark that the result depends essentially on the
Standing Assumptions~\ref{StandingAssumption} as seen in the proof.
\begin{proof}
The proof is an extension of
Lemma~\ref{TheAverageIsSolutionInSolutionSpace}. More precisely, we
need to show that {\rm (i)} the functions $x(t) = \sbm{\bar
  \phi(t,\cdot) \\ \rho \bar \phi_t(t,\cdot)}$ and $\sbm{\bar u\\ f}$
satisfy $x(0) \in \Xscr_w$, $\sbm{\bar u\\ f} \in
L^2(\rplus;\Uscr_w)$, and $x \in C(\rpluscl;\Xscr_w) \cap
W_{loc}^{1,1}(\rplus;(\Xscr_w)_{-1})$; and that {\rm (ii)} the
dynamical equations \eqref{WebsterNodeDynamics} are satisfied with
$\tilde y = \bar y$ where the Hilbert spaces $\Uscr_w$, $\Xscr_w$,
$(\Xscr_w)_{-1}$ and the system node $S^{(W)}$ are defined in
Section~\ref{WebsterDirectSec}.

Now, it is immediate from assumptions that $\sbm{ \bar u\\ f} \in
L^2(\rplus;\Uscr_w)$. Recalling that $\Zscr_w \subset \Xscr_w$, the
inclusion $x(0) \in \Xscr_w$ follows because a stronger result $x(t)
\in \Zscr_w$ for all $t \geq 0$ has been shown in the proof of
Lemma~\ref{TheAverageIsSolutionInSolutionSpace}.

We work under the regularity assumptions \eqref{WaveEqSolutionRegular}
on the classical solution $\phi$ of
\eqref{IntroWaveEq}--\eqref{IntroWaveObs}, and hence $\sbm{ \phi
  \\ \rho \phi_t} \in C^1(\rplus; H^1(\Omega) \times L^2(\Omega))$.
The averaging operator $\mathcal A$ defined by
\eqref{AveragedWaveEqSol} satisfies $\mathcal A \in \BLO(H^k(\Omega);
H^k(0,1))$ for all $k \geq 0$ by \cite[Proposition~5.3]{L-M:WECAD}.
Thus, the averaged solution $\bar \phi(t, \cdot) = \mathcal A \phi(t,
\cdot)$ satisfies $x =\sbm{\bar \phi \\ \rho \bar \phi_t} \in
C^1(\rplus; \Xscr_w)$ since $\Xscr_w = H^1_{\{1\}}(0,1) \times
L^2(0,1)$ and $\bar \phi(t,1) = 0$ follows from the boundary condition
$\phi(t,\br) = 0$ for all $\br \in \Gamma(1)$. Because $\Xscr_w
\subset (\Xscr_w)_{-1}$ with a continuous embedding (see
Definition~\ref{SysNodeDef}), we have $x \in C(\rpluscl;\Xscr_w) \cap
W_{loc}^{1,1}(\rplus;(\Xscr_w)_{-1})$ as required.

Let us first check the top row of \eqref{WebsterNodeDynamics}; i.e.,
\begin{equation} \label{TopRow}
 x'(t) = [\AB]_w \sbm{x(t) \\ \bar
   u(t)\\ f(t,\cdot)} = A_{-1} x(t) + B \bar u(t) + \sbm{0 \\ \rho f (t,\cdot)}
\end{equation}
where the operators $A_{-1}, B$ are as in
Section~\ref{WebsterDirectSec}.  Since $x(t) \in \Zscr_w$ (as is
already stated in this proof) and $A_{-1}\rst{\Zscr_w} = L_w - B G_w$,
we conclude that $A_{-1} x(t) + B u(t) + \sbm{0 \\ \rho f (t,\cdot)} =
L_w x(t) + B \left (\bar u(t) - G_w x(t) \right ) + \sbm{0 \\ \rho f
  (t,\cdot)} = x'(t) + B \left (\bar u(t) - G_w x(t) \right )$ by
Lemma~\ref{TheAverageIsSolutionInSolutionSpace}.  Thus, equation
\eqref{TopRow} holds since $\bar u(t) = G_w x(t)$ follows from the
second equation in \eqref{IntroWaveEq} as explained in
\cite[Eqs.~(3.6) and (3.8), as shown at the end of
  Section~4]{L-M:WECAD}, noting that the last two condition listed in
Standing Assumptions~\ref{StandingAssumption} hold.

It remains to treat the bottom row of \eqref{WebsterNodeDynamics}
which takes the form 
\begin{equation*}
  \tilde y(t) = [\CD]_w \sbm{x(t) \\ \bar u(t)\\ f(t,\cdot)} = K_w  x(t).
\end{equation*}
Similarly as above for the input equation $\bar u(t) = G_w x(t)$, we
observe that $\bar y(t) = K_w x(t)$ as well. Hence $\tilde y = \bar y$
follows, and the proof is complete.
\end{proof}

For the rest of the section, we denote by $\sbm{\psi \\ \rho \psi_t}$
the unique solution of \eqref{WebsterNodeDynamics} with $f \equiv 0$
and output $\tilde y$, referring to the left panel in
Fig.~\ref{FigBoth}. By Lemma~\ref{LiftingWeakSolutionToStrongLemma},
the function $\sbm{\bar \phi \\ \rho \bar \phi_t}$ is the unique
solution of \eqref{WebsterNodeDynamics} with $f = F + G + H \in
C(\rpluscl;L^2(0,1))$ and output $\bar y$, referring to the right
panel in Fig.~\ref{FigBoth}.  By subtracting the model equations for
$\psi$ and $\bar \phi$ from each other, we get the equations for the
tracking error. Indeed, because both $\sbm{\psi \\ \rho \psi_t}$ and
$\sbm{\bar \phi \\ \rho \bar \phi_t}$ are strong solutions in the
sense of Corollary~\ref{WebsterNodeIsWellPosedCor}, the tracking error
$\tilde e := \sbm{e \\ \rho e_t} = \sbm{\psi - \bar \phi \\ \rho (\psi
  - \bar \phi)_t}$ is the unique strong solution of the \emph{tracking
  error model}
\begin{equation}\label{TrackingErrorModel}
 \bbm{\tilde e'(t) \\ \tilde y(t) - \bar y(t)} = S^{(W)} \sbm{\tilde
   e(t) \\ 0 \\ F(t,\cdot) + G(t,\cdot) + H(t,\cdot)} \quad \text{ and } \quad \quad \tilde e(0) = \bbm{0 \\ 0}.
 \end{equation}
Now, the tracking error can be estimated for $T \geq 0$ by using the
well-posedness estimate \eqref{WebsterWellPosedInEq2} for strong
solutions, given in Theorem~\ref{WebsterNodeIsWellPosedThm}:
\begin{equation}
  \label{FirstErrorEstimate}
  \begin{aligned}
    & \norm{\tilde e(T)}_{\Xscr_w}^2 + \norm{\tilde y - \bar y}_{L^2((0,T);\Yscr_w)}^2 \\ 
    & \leq 5 (\rho + 1)^{1/2} (T + 1) \cdot  \norm{F + G + H}^2_{L^2((0,T);L^2(0,1))}.
  \end{aligned}
\end{equation}

It remains to translate \eqref{FirstErrorEstimate} to our first
\emph{a posteriori} estimate recalling the norm of $\Xscr_w$ in
\eqref{WebstersEnergyNorm} that was used for deriving
\eqref{FirstErrorEstimate}.
\begin{theorem}
  \label{MainTheorem2}
  Let the sets $\Omega$, $\Gamma$, and $\Gamma(s)$ for $s \in [0,1]$
  be defined as in Section~\ref{IntroSec}, and assume that the
  Standing Assumptions~\ref{StandingAssumption-1} and
  \ref{StandingAssumption} hold.  Moreover, assume the following:
\begin{enumerate}
\item   \label{MainTheorem2FirstAss}
  Let $u \in C^2(\rpluscl; L^2(\Gamma(0)))$, and define its spatial
  average $\bar u$ by \eqref{AveragedInputFunction}. 
\item \label{MainTheorem2CompatibilityAss} Let $\phi_0 \in
  H^1(\Omega)$ with $\phi\rst{\Gamma(0)} = 0$, $\Delta \phi_0 \in
  L^2(\Omega)$, and $\frac{\partial \phi_0}{\partial
    \bnu}\rst{\Gamma(0)\cup\Gamma} \in L^2(\Gamma(0)\cup\Gamma)$. Let
  $p_0 \in H^1(\Omega)$ with $\phi\rst{\Gamma(0)} = 0$, and assume
  that the compatibility condition with the input function $u$ holds:
  \begin{equation*}
    c \frac{\partial \phi_0}{\partial \bnu}(\br) + \rho^{-1} p_0(\br)
    = 2 \sqrt{\tfrac{c}{\rho A(0)}} \, u(0,\br) \quad \text{ for all }
    \quad \br \in \Gamma(0).
  \end{equation*}

\item By $\phi:\rpluscl \times \Omega \to \R$ denote the
  solution\footnote{As explained in \cite[Theorem~5.1]{A-L-M:AWGIDDS}
    for $\alpha > 0$ and \cite[Corollary~5.2]{A-L-M:AWGIDDS} for
    $\alpha = 0$.} of the wave equation model \eqref{IntroWaveEq}
  Define the output $y$ by \eqref{IntroWaveObs}.
\item Define the spatially averaged version $\bar \phi$ of $\phi$ by
  \eqref{AveragedWaveEqSol}. Similarly with $\bar u$, define $\bar y$
  in terms of $y$.
\item  \label{MainTheorem2LastAss} By $\psi:\rpluscl \times [0,1] \to \R$ denote the
  solution\footnote{As explained in
    \cite[Theorem~4.1]{A-L-M:AWGIDDS}.}  of the generalised Webster's
  equation \eqref{IntroWebstersEq} with the input $\tilde u = \bar u$,
  and define the output $\tilde y$ by \eqref{IntroWebstersEqObs}.
\end{enumerate}
  Then the tracking error $e = \psi - \bar \phi$, as described by the
  left panel of Fig.~\ref{FigBoth}, is bounded from above for all $T
  \geq 0$ by the inequality
  \begin{align*}
    \norm{\left ( \psi - \bar \phi \right ) (T,\cdot) }_{H^1(0,1)} & +
     \norm{\left ( \psi_t - \bar \phi_t \right ) (T,\cdot) }_{L^2(0,1)} 
    + \norm{\tilde y - \bar y}_{L^2(0,T)}
\\ & \leq 4 C_{\Omega} \rho^{-1/2} (\rho + 1)^{3/4} (T+ 1)^{1/2} \norm{F + G + H}_{L^2((0,T) \times (0,1))}
  \end{align*}
 where the constant $C_{\Omega}$ given by
\begin{equation} \label{MainTheorem2ConstantDef}
  C_{\Omega}^2 =  \frac{2}{\min_{s \in [0,1]}(A(s),
A(s)/c(s)^2)} + 1, \quad c(s) = c \Sigma(s), 
\end{equation}
depends only on the geometry of $\Omega$, and the functions $F$, $G$,
and $H$ are given by \eqref{ControlTermsF} -- \eqref{ControlTermsH} in
terms of solution $\phi$ of \eqref{IntroWaveEq} and the problem data.
\end{theorem}
\begin{proof}
We observe that $\phi$ has the regularity required in
\eqref{StandingSmoothness} since it is part of
\cite[Theorem~5.1]{A-L-M:AWGIDDS} for classical solutions. Hence, all
that has been stated above about $\bar \phi$ is at our disposal.
Recalling the energy norm \eqref{WebstersEnergyNorm} of $\Xscr_w$, we
get
\begin{equation*}
\begin{aligned}
  & \norm{\left ( \psi_s - \bar \phi_s \right )(T, \cdot)}^2_{L^2(0,1)} + \norm{\left (\psi_t - \bar \phi_t \right )(T, \cdot)}^2_{L^2(0,1)} \\
& \leq C_1 \norm{\sbm{(\psi - \bar \phi)(T, \cdot) \\ \rho (\psi_t
  - \bar \phi_t)(T, \cdot)}}^2_{\Xscr_w}
\end{aligned}
\end{equation*}
where $C_1^{-1} := \rho/2 \cdot \min_{s \in [0,1]}(A(s),
A(s)/c(s)^2)$, $c(s) = c \Sigma(s)$. Thus, using \eqref{FirstErrorEstimate}, we get
\begin{equation}
\begin{aligned}
 & \norm{\left ( \psi - \bar \phi \right )(T, \cdot)}^2_{H^1(0,1)} + \norm{\left (\psi_t - \bar \phi_t \right )(T, \cdot)}^2_{L^2(0,1)}
+  \norm{\tilde y - \bar y}^2_{L^2(0,T)} \\
& \leq C_2 \left (\norm{\sbm{(\psi - \bar \phi)(T, \cdot) \\ \rho (\psi_t
  - \bar \phi_t)(T, \cdot)}}^2_{\Xscr_w} + \norm{\tilde y - \bar y}^2_{L^2(0,T)} \right ) \\
& \leq  5 C_2 (\rho + 1)^{1/2} (T + 1)
\cdot  \norm{F + G + H}^2_{L^2((0,T);L^2(0,1))}
\end{aligned}
\end{equation}
where $C_2 = C_1 + 1$. Taking the square root of both sides and using
$(a + b+ c)^2 \leq 3 \left (a^2 + b^2 + c^2 \right )$ gives
\begin{equation*}
  \begin{aligned}
    & \norm{\left ( \psi - \bar \phi \right ) (T,\cdot) }_{H^1(0,1)} +
    \norm{\left ( \psi_t - \bar \phi_t \right ) (T,\cdot) }_{L^2(0,1)}
    + \norm{\tilde y - \bar y}_{L^2(0,T)} \\
    &  \leq \sqrt{15 C_2} (\rho +
    1)^{1/4} (T+ 1)^{1/2} \norm{F + G + H}_{L^2((0,T) \times (0,1))}
  \end{aligned}
\end{equation*}
which gives the claim.
\end{proof}

\section{\label{PosteriorSection}A posteriori error estimate}

In essence, the following Theorem~\ref{MainTheorem3} follows from
Theorem~\ref{MainTheorem2} by estimating the functions $F$, $G$, and
$H$ in terms of $\phi$ and the problem data. 

 By $f\rst{\Gamma}$ denote the Dirichlet trace of a function $f$
 defined on $\Omega$. Define the Hilbert space
\begin{equation*}
  E(\Omega) = \left \{ f \in H^1(\Omega) : \Delta f \in L^2(\Omega),
  f\rst{\Gamma} \in H^1(\Gamma) \right \},
\end{equation*}
equipped with the norm
\begin{equation*}
  \norm{f}^2_{E(\Omega)} = \norm{f}^2_{H^1(\Omega)} + \norm{\Delta f}^2_{L^2(\Omega)} + 
  \norm{f\rst{\Gamma}}^2_{H^1(\Gamma)}.
\end{equation*}
Recall that $\Gamma \subset \partial \Omega$ denotes the walls of the tube
$\Omega$, excluding the ends $\Gamma(0)$ and $\Gamma(1)$.

\begin{theorem}
  \label{MainTheorem3}
  Let the sets $\Omega$, $\Gamma$, and $\Gamma(s)$ for $s \in [0,1]$
  be defined as in Section~\ref{IntroSec}, and assume that the
  Standing Assumptions~\ref{StandingAssumption-1} and
  \ref{StandingAssumption} hold.  Make the same assumptions
  \eqref{MainTheorem2FirstAss}--\eqref{MainTheorem2LastAss} on signals
  $u$, $\bar u$, $y$, $\bar y$, $\tilde y$, and solutions $\phi$,
  $\bar \phi$ as in Theorem~\ref{MainTheorem2}. Denote by $\bar{\bar
    \phi}$ the extension of the averaged solution $\bar \phi$ to all
  of $\clos{\Omega}$, given by
\begin{equation} \label{DilationOfAverage}
  \bar{\bar \phi} (\cdot, \br) = \bar \phi(\cdot,s) \quad \text{ for all } \quad \br
  \in \clos{\Gamma(s)}, \quad s \in [0,1].
\end{equation}
  Then the tracking error $e = \psi - \bar \phi$, as described by the
  left panel of Fig.~\ref{FigBoth}, is bounded  by the inequality
  \begin{multline*}
\norm{\left ( \psi - \bar \phi \right ) (T,\cdot) }_{H^1(0,1)}  +
    \norm{\left ( \psi_t - \bar \phi_t \right ) (T,\cdot) }_{L^2(0,1)} 
    + \norm{\tilde y - \bar y}_{L^2(0,T)}\\
  \begin{aligned}
     \leq & 7 C_{\Omega} \rho^{-1/2} (\rho + 1)^{3/4} (T+ 1)^{1/2} \\
     \cdot & \left ( 
     \left ( \norm{A'}_{L^\infty(0,1)} 
    + \norm{A''}_{L^\infty(0,1)} \right ) C_{\mathcal F} \left \|\phi - \bar{\bar \phi} \right \|_{L^2((0,T);E(\Omega))}  \right . \\
    & + \norm{\max{(\kappa, \kappa')}}^2_{L^\infty(0,1)}  C_{\mathcal H, 1} \, \left \| \nabla \left ( \phi - \bar{\bar \phi}\right ) \right \|_{L^2([0,T] \times \Omega; \C^3)} \\
    &  +  \alpha \left ( C_{\mathcal G} +  C_{\mathcal H , 2} \right )
     \left \| \left (\phi - \bar{\bar  \phi} \right )_t \right \|_{L^2((0,T);H^1(\Omega))}  \\
    & \left . +   \norm{\kappa}_{L^\infty(0,1)} C_{\mathcal H, 3}  \left \|  \Delta \phi - \overline{\overline{\Delta \phi}}  \right \|_ {L^2([0,T] \times \Omega)}  \right )   
  \end{aligned}
  \end{multline*}
 for all $T \geq 0$ where
\begin{equation*}
  \left ( \overline{\overline{\Delta \phi}} \right )(t,\br) := \mathcal A(\Delta \phi (t,\cdot))(s) \quad  \text{for all} \quad  \br \in \Gamma(s), \quad s \in [0,1], 
  \quad \text{and} \quad  t \in \rpluscl, 
\end{equation*}
the constant $C_{\Omega}$ is given by
 \eqref{MainTheorem2ConstantDef}, the constants $C_{\mathcal F}$ and
 $C_{\mathcal G}$ are as given in Proposition~\ref{FGProp}, and
the constants $C_{\mathcal H , 1}$, $C_{\mathcal H , 2}$, and $C_{\mathcal H , 3}$
are as given in Proposition~\ref{HProp}.

All of the constants on the right hand side depend only on the domain
$\Omega$. 
\end{theorem}
\noindent The proof of this theorem is divided into
Propositions~\ref{FopGopBoundedProp}, \ref{FGProp}, and
\ref{HProp}.  Even though the constants in
Theorem~\ref{MainTheorem3} depend only on the domain $\Omega$, their
numerical values are difficult to obtain since they contain, e.g.,
norms of trace mappings.

For $f \in E(\Omega)$ and $g \in H^{1}(\Omega)$, define the linear operators
\begin{equation} \label{FopGopDef}
  \begin{aligned}
    & (\mathcal F f)(s) = 
    A''(s)(\mathcal{A}f - \mathcal{B} f\rst{\Gamma})+A'(s)\frac{\partial}{\partial s}
    (\mathcal{A} f-\mathcal{B} f \rst{\Gamma}) \\
& (\mathcal G g)(s) =  -\frac{2\pi W(s)}{A(s)} \left(\mathcal{A} g -\mathcal{B}g\rst{\Gamma}\right)
  \end{aligned}
\end{equation}
where we use the two averaging operators that have been introduced in \cite{L-M:WECAD}
\begin{equation*}
  (\mathcal A f)(s) := \frac{1}{A(s)} {\int_{ \Gamma(s) }{f \, d A}}\quad \text{ and } 
  (\mathcal B g\rst{\Gamma})(s) := \frac{1}{2 \pi}\int_{0}^{2 \pi} g(s,R(s),\theta) \, d \theta \quad
\end{equation*}
for all $s \in (0,1)$.  
\begin{proposition} \label{FopGopBoundedProp}
  The operators defined in \eqref{FopGopDef} satisfy $\mathcal F \in
  \BLO(E(\Omega);L^2(0,1))$  with the estimate
\begin{equation*}
     \norm{\mathcal F f}_{L^2(0,1)} \leq 
    \left ( \norm{A'}_{L^\infty(0,1)} + \norm{A''}_{L^\infty(0,1)} \right ) C_{\mathcal F} 
    \norm{f}_{E(\Omega)}  
\end{equation*}
and $\mathcal G \in \BLO(E(\Omega);H^1(0,1)) \cap
\BLO(H^1(\Omega);H^{1/2}(0,1))$.

Moreover, $\Null{\mathcal G} \cap E(\Omega) \subset \Null{\mathcal F}$,
$\Null{\mathcal G}$ is closed in $H^1(\Omega)$, and $\Null{\mathcal
  F}$ closed in $E(\Omega)$.  If $f \in H^1(\Omega)$, and $\bar{\bar
  f}$ is defined by dilation
\begin{equation*}
  \bar{\bar f} (\br) = (\mathcal A f) (s) \quad \text{ for all } \quad
  \br \in \Gamma(s) \quad \text{ and } \quad s \in (0,1)
\end{equation*}  
then $\bar{\bar f} \in \Null{\mathcal G}$.  Similarly, if $f \in
E(\Omega)$ is such that $\mathcal Af \in H^2(0,1)$ then $\bar{\bar f}
\in \Null{\mathcal G} \cap E(\Omega)$.
\end{proposition}
\begin{proof}
As shown in \cite[Propositions~5.2 and 5.3]{L-M:WECAD}, we have 
\begin{equation*}
  \mathcal A \in \BLO(L^2(\Omega); L^2(0,1))
  \cap\BLO(H^1(\Omega);H^1(0,1)) \text{ and } \mathcal B
  \in\BLO(H^{s}(\Gamma);H^{s}(0,1))
\end{equation*}
for all $s \in \R$.  Because the functions $A(\cdot)$ and $W(\cdot)$
are smooth and strictly positive, the norm estimates for $\mathcal F$
and $\mathcal G$ follow.  The claims about the null spaces are
evident, apart from the last one.  

 Since $\bar{\bar f}$ is constant on each $\Gamma(s)$ for $s \in
 (0,1)$, it would follow that the two averages $(\mathcal A \bar{\bar
   f})(s)$ and $(\mathcal B \bar{\bar f}\rst{\Gamma})(s)$ would
 coincide for all $s$.  Thus, formally $\bar{\bar f} \in
 \Null{\mathcal G}$.  It remains to show that $\bar{\bar f} \in
 H^1(\Omega)$ if $f \in H^1(\Omega)$, and that $\bar{\bar f} \in
 E(\Omega)$ if $f \in E(\Omega)$

  We choose a smooth curve $\bl:[0,1] \to \Gamma$ on the tube wall
  such that $\Gamma(s) \cap \bl$ consists of a single point. The cut
  the tube wall open along $\bl$, and map the surface $\Gamma
  \setminus \bl$ to the unit square $[0,1] \times (0,1)$ by a smooth
  diffeomorphism, so that the circles $\partial \Gamma(s) \setminus \{
  \bl(s) \}$ map onto $\{ s \} \times (0,1)$.  Now, it it clear that
  the extension $\tilde f(s,\xi) := \bar f(s)$ for $\xi \in (0,1)$
  satisfies $\tilde f \in H^1((0,1)^2)$ because $f \in H^1(\Omega)$
  implies $\bar f = \mathcal A f \in H^1(0,1)$ by
  \cite[Proposition~5.2]{L-M:WECAD}. Thus $\bar{\bar f} \in
  H^1(\Gamma)$ by pullback. By a similar argument, we have $\bar{\bar
    f} \in H^2(\Omega) \subset E(\Omega)$ if $\bar f \in H^2(0,1)$
  which completes the proof.
\end{proof}

With the help of the operators $\mathcal F$ and $\mathcal G$, the
forcing terms $F$ and $G$ given in \eqref{ControlTermsF} --\eqref{ControlTermsG} may be
written as
\begin{equation} \label{FGDef}
  F(t,s)  = (\mathcal F \phi(t,\cdot))(s)  
\quad \text{ and } \quad
  G(t,s)  =\alpha (\mathcal G \phi_t(t,\cdot))(s) \quad
\end{equation}
for all $(t,s) \in \rpluscl \times (0,1)$.
\begin{proposition} \label{FGProp}
  Make the same assumption as in Theorem~\ref{MainTheorem3}.  The
  forcing functions $F$ and $G$, given by
  \eqref{ControlTermsF}--\eqref{ControlTermsG}, satisfy the estimates
  \begin{align*}
    \norm{F(t,\cdot)}_{L^2(0,1)} & \leq  \left (
    \norm{A'}_{L^\infty(0,1)} + \norm{A''}_{L^\infty(0,1)} \right ) C_{\mathcal F} \,
    \left \| \left ( \phi - \bar{\bar \phi}\right ) (t,\cdot)
    \right \|_{E(\Omega)} \text{ and } \\ \norm{G(t,\cdot)}_{L^2(0,1)} & \leq
    \alpha C_{\mathcal G} \, \left \| \left ( \phi - \bar{\bar
        \phi} \right )_t (t,\cdot) \right \| _{H^1(\Omega)}
  \end{align*}
for all $t \geq 0$ where the constants $C_{\mathcal F}$ and $C_{\mathcal G}$ depend only
on the geometry of $\Omega$.
\end{proposition}
\begin{proof}
As shown in \cite[Theorem~5.1 and Corollary~5.2]{A-L-M:AWGIDDS}, the
unique classical solution $\phi$ of the wave equation
\eqref{IntroWaveEq} satisfies the regularity assumptions
\eqref{WaveEqSolutionRegular}. Hence, all results of
\cite[Section~5]{L-M:WECAD} can be used: in particular, that
$\phi\rst{\Gamma} \in C(\R_+;H^1(\Gamma))$ by
\cite[Proposition~5.1]{L-M:WECAD} and, hence, $\phi(t,\cdot) \in
E(\Omega)$ for all fixed $t \in \rpluscl$.  Further, by
Lemma~\ref{ALittleBitMoreRegularityLemma} we get that $\bar \phi \in
H^2(0,1)$.  Together with \eqref{FGDef}, we get
  \begin{align*}
  & \norm{F(t,\cdot)}_{L^2(0,1)} 
    \leq  \left \| \mathcal F \left ( \phi(t,\cdot) 
      - \bar{\bar \phi}(t,\cdot) \right ) \right \|_{L^2(0,1)}  
    + \norm{\mathcal F \bar{\bar \phi}(t,\cdot)}_{L^2(0,1)} \\
& \leq  C_{\mathcal F} 
    \left ( \norm{A'}_{L^\infty(0,1)} + \norm{A''}_{L^\infty(0,1)} \right ) \cdot
    \left \| \left ( \phi - \bar{\bar \phi}\right ) (t,\cdot) \right \|_{E(\Omega)}
\end{align*}
because $\mathcal F \bar{\bar \phi} \equiv 0$ by
Proposition~\ref{FopGopBoundedProp}. The estimate involving $G$ is
proved similarly, noting that the dissipativity constant $\alpha$ is
always nonnegative.
\end{proof}

It remains to treat the term $H(\cdot)$ given in \eqref{ControlTermsH}:
\begin{proposition} \label{HProp}
  Make the same assumption as in Theorem~\ref{MainTheorem3}. Then the
  forcing function $H$, given by \eqref{ControlTermsH}, satisfies the
  estimate
\begin{equation} \label{HPropEq1}
\begin{aligned}
& \norm{H(t,\cdot)}_{L^2(0,1)} \leq 
\norm{\max{(\kappa, \kappa')}}_{L^\infty(0,1)} 
C_{\mathcal H, 1} \, \left \| \nabla \left ( \phi - \bar{\bar \phi}\right )(t,\cdot)\right \|_{L^2(\Omega; \C^3)} \\
& + \alpha C_{\mathcal H , 2} \left \| \left ( \phi - {\bar{\bar \phi}} \right )_t (t,\cdot) \right \|_{H^1(\Omega)}
+ \norm{\kappa}_{L^\infty(0,1)} C_{\mathcal H, 3}  \left \| \left (\Delta \phi - \overline{\overline
  {\Delta \phi}}  \right )(t,\cdot)  \right \|_ {L^2(\Omega)}
\end{aligned}
\end{equation}
for all $t \geq 0$ where the constants $C_{\mathcal H, 1}$,
$C_{\mathcal H, 2}$, and $C_{\mathcal H, 3}$ depend only on the
geometry of $\Omega$.
\end{proposition}
\begin{proof}
Let us begin with the first term in $H$ in \eqref{ControlTermsH}.
Denoting by $\bar{\bar \phi}$ the extension given in
\eqref{DilationOfAverage}, we observe by using the gradient formula in
\cite[Section~2]{L-M:WECAD} that $\nabla \bar{\bar \phi} = \bt(s) \,
\Xi \frac{\partial \bar{\bar \phi}}{\partial s}$ and $\nabla \left (
\Xi^{-1} \right ) = - \bt(s) \, r \kappa'(s) \cos{\theta} - \bn(s)
\,\kappa(s)$. Thus, recalling that  $\bar{\bar \phi} = \bar{\bar \phi}(t,s)$, we get
\begin{equation*}
  \int_{\Gamma(s)} {\nabla \left( \frac{1}{\Xi}\right)\cdot\nabla \bar{\bar \phi}  \, \frac{dA}{\Xi} }
= - \kappa'(s) \frac{\partial \bar{\bar \phi}}{\partial s}  \int_{\Gamma(s)} {r \cos{\theta} \, dA} = 0
\end{equation*}
where $dA = r dr d \theta$.  Hence, we get by using H\"older's
inequality
\begin{align*}
  & \left \| \int_{\Gamma(s)} { \nabla \left( \frac{1}{\Xi}\right)\cdot\nabla \phi  
    \, \frac{dA}{\Xi}} \right \| _{L^2(0,1)}^2  = 
\left \| \int_{\Gamma(s)} { \nabla \left( \frac{1}{\Xi}\right)\cdot\nabla \left ( \phi - \bar{\bar \phi}\right )  
    \, \frac{dA}{\Xi}} \right \| _{L^2(0,1)}^2 
\\
    \leq & \int_0^1{\left (
      \int_{\Gamma(s)}{\abs{\nabla \left(\frac{1}{\Xi}\right)}^2 \frac{dA}{\Xi}}
      \cdot
      \int_{\Gamma(s)}{\abs{\nabla \left ( \phi - \bar{\bar \phi}\right )}^2\frac{dA}{\Xi}} 
      \right ) \,ds } \\ 
    \leq & \norm{\max{(\kappa, \kappa')}}^2_{L^\infty(0,1)} \,  C_{\mathcal H, 1}^2 
    \cdot \int_0^1{ \left (
      \int_{\Gamma(s)}{\abs{\nabla \left ( \phi - \bar{\bar \phi}\right ) }^2\frac{dA}{\Xi}}  
      \right ) } \, ds  \\
     = &  \norm{\max{(\kappa, \kappa')}}^2_{L^\infty(0,1)} C_{\mathcal H, 1}^2 \cdot \left \| \nabla \left ( \phi - \bar{\bar \phi}\right )  \right \|_{L^2(\Omega;\C^3)}^2
\end{align*}
where $C_{\mathcal H, 1} := \sup_{s \in [0,1]}\left(\int_{\Gamma(s)}{
  (r + 1)^2 \Xi^{-1} \,dA} \right)^{1/2}$,  since $dV
= \Xi^{-1} \, dA \, ds$ and $\abs{\nabla
  \left(\Xi^{-1}\right)} \leq \max{(\kappa, \kappa')} (r + 1)$.

Let us estimate next the last term in \eqref{ControlTermsH}. Because
the function ${\bar{\bar \phi}}_t \equiv \overline{\overline
  {\phi_t}}$ does not depend on $r$ and $\theta$ variables at all, we
have
\begin{equation*}
\begin{aligned}
  & \left \| \int_{0}^{2\pi}{
    \phi_t(t, \cdot,R(\cdot),\theta) \cos{\theta} d\theta } \right \|_{L^2(0,1)} \\
  & \leq 
  \left \| \int_{0}^{2\pi}{
    \left (\phi_t - {\bar{\bar \phi}}_t \right ) (t,\cdot,R(\cdot),\theta)
    \cos{\theta} \, d\theta } \right \|_{L^2(0,1)}
  +   \left \| {\bar{\bar \phi}}_t(t,\cdot) \cdot 
  \int_{0}^{2\pi}{\cos{\theta} \, d\theta } \right \|_{L^2(0,1)} \\
  & = \left \| \int_{0}^{2\pi}{
    \left (\phi_t - {\bar{\bar \phi}}_t \right ) (t,\cdot,R(\cdot),\theta)
    \cos{\theta} \, d\theta } \right \|_{L^2(0,1)}.
\end{aligned}
\end{equation*}
Since the surface element on tube wall $\Gamma$ is given by $dS = W(s)
\, d\theta \, ds$ by \cite[Section~2]{L-M:WECAD}, we get for all $t
\geq 0$
\begin{align*}
& \left \| \frac{\alpha W(\cdot)
    \eta(\cdot)}{A(\cdot)} \int_{0}^{2\pi}{ 
      \phi_ t (t, \cdot,R(\cdot),\theta) \cos{\theta} d\theta } \right \|_{L^2(0,1)}^2 \\
& \leq  \alpha^2  C_3^2 
\int_{0}^1{ 
\left |  \int_{0}^{2\pi}{ W(s)^{1/2} 
\left (\phi_t - {\bar{\bar \phi}}_t \right ) (t,s,R(s),\theta) \cos{\theta} \, d\theta } \right | ^2
    \, ds} \\
& \leq \pi \alpha^2 C_3^2
\int_{0}^1{ 
\int_{0}^{2\pi}{  \left | \left (\phi_t - {\bar{\bar \phi}}_t \right ) (t,s,R(s),\theta)\right |^2  W(s) \, d\theta } 
    \, ds} \\
& = \pi \alpha^2  C_3^2
\int_{\Gamma}{\left | \left ( \phi_t(t,\cdot) - {\bar{\bar \phi}}_t(t,\cdot) \right )\rst{\Gamma}  \right |^2  \, dS }  
= \pi \alpha^2 C_3^2 \left \|  \left ( \phi_t(t,\cdot) - {\bar{\bar \phi}}_t(t,\cdot) \right ) \rst{\Gamma} \right \|^2_{L^2(\Gamma)} \\
& \leq \pi \alpha^2 C_3^2 C_4^2 \left \|  \left ( \phi_t(t,\cdot) - {\bar{\bar \phi}}_t(t,\cdot) \right )\rst{\Gamma} \right \|^2_{H^{1/2}(\Gamma)}
\leq  \alpha^2 C_{\mathcal H , 2}^2 \left \| \left ( \phi - {\bar{\bar \phi}} \right )_t (t,\cdot) \right \|^2_{H^1(\Omega)}
\end{align*}
where $C_3 := \sup_{s \in [0,1]}{\frac{\eta(s)}{A(s)}}$, $C_{\mathcal
  H , 2} := \pi^{1/2} C_3 C_4 C_5$, and the constants $C_4, C_5$ are
the norms of the inclusion $H^{1/2}(\Gamma) \subset L^{2}(\Gamma)$ and
the trace mapping from $H^1(\Omega)$ into $H^{1/2}(\Gamma)$,
respectively.

It remains to treat the second term in $H$ in \eqref{ControlTermsH}.
We first observe that the error function $E = E(s,r,\theta)$
introduced in \eqref{SSCFError} averages to zero over each
intersectional surface $\Gamma(s)$. We have
\begin{equation*}
\begin{aligned}
  & \int_{\Gamma(s)}{E(s,r,\theta) \, dA} \\
  & = - 2 \kappa(s) \int_{\Gamma(s)}{r \cos{\theta} \, dA}
+ \kappa(s)^2 \int_{0}^{R(s)} \int_{0}^{2 \pi}  {\left (r^2 \cos^2{\theta} - \tfrac{R(s)^2}{4} \right ) \, r dr d \theta} \\
& = \kappa(s)^2 \left (  \int_{0}^{R(s)}{\, r^3 dr}  \int_{0}^{2 \pi} { \cos^2{\theta} \, d \theta} - \frac{R(s)^2}{4} \int_{0}^{R(s)}{\, r dr}  \int_{0}^{2 \pi} {d \theta}  \right ) \\
& = \kappa(s)^2 \left ( \frac{1}{4} R(s)^4 \cdot \int_{0}^{2 \pi} { \cos^2{\theta} \, d \theta}  - \frac{R(s)^2}{4} \cdot \frac{1}{2} R(s)^2 \cdot 2 \pi \right )= 0
\end{aligned}
\end{equation*}
because $\int_{0}^{2 \pi} { \cos^2{\theta} \, d \theta} = \pi$.
Considering now the second time derivative $\bar{\bar \phi}_{tt} =
\overline{\overline{\phi_{tt}}}$ of $\bar{ \bar \phi}$ in
\eqref{DilationOfAverage}, we see that also $\bar{\bar \phi}_{tt}$
does not depend on the variables $r$ and $\theta$ at all. 
Recalling that $\phi$ satisfies the wave
equation $\Delta\phi = c^{-2}\phi_{tt}$, we get 
\begin{equation*}
  c^2 \int_{\Gamma(s)}{E \,  \overline{\overline{\Delta \phi}} (t,s)   \, dA} =
  \int_{\Gamma(s)}{E \, \bar{\bar \phi}_{tt}(t,s)   \, dA} = \bar{\bar \phi}_{tt}(t,s) 
  \int_{\Gamma(s)}{ E  \, dA} = 0
\end{equation*}
for all $s \in [0,1]$. Hence, we get the estimate
\begin{align*}
 &  \left \| \frac{1}{A(s)}\int_{\Gamma(s)}{ E \Delta\phi \, dA}  \right \|_{L^2(0,1)}  =
  \left \| \frac{1}{A(s)}\int_{\Gamma(s)}{ E \cdot \left ( \Delta\phi - \overline{\overline{\Delta \phi}}   \right ) \, dA} \right \|_{L^2(0,1)} \\
&  =  \norm{\kappa}_{L^\infty(0,1)} \left \|
\mathcal{A} \left ( \frac{E}{\kappa} \cdot \left ( \Delta\phi - \overline{\overline{\Delta \phi}}   \right ) \right  )\right \|_{L^2(0,1)} 
 \leq  \norm{\kappa}_{L^\infty(0,1)} C_{\mathcal H, 3}  \,  \left \| \Delta\phi - \overline{\overline{\Delta \phi}}  \right \|_ {L^2(\Omega)}
\end{align*}
where $C_{\mathcal H, 3} :=
\norm{\mathcal{A}}_{\BLO(L^2(\Omega);L^2(0,1))}
\norm{E/\kappa}_{L^\infty(\Omega)}$; the boundedness of $\mathcal{A}$
is by \cite[Propositions~5.2]{L-M:WECAD}.
\end{proof}



\begin{thebibliography}{10}

\bibitem{A-A:LOGMNFMCAL}
A.~Aalto.
\newblock A low-order glottis model with nonturbulent flow and mechanically
  coupled acoustic load.
\newblock Master's thesis, Helsinki University of Technology, 2009.

\bibitem{A-L-M:AWGIDDS}
A.~Aalto, T.~Lukkari, and J.~Malinen.
\newblock Acoustic wave guides as infinite-dimensional dynamical systems.
\newblock {\em ESAIM: Control Optim. Calc. Var. (to
  appear)}, 2014.

\bibitem{A-A-H-J-K-K-L-M-M-P-S-V:LSDASMRIS}
D.~Aalto, O.~Aaltonen, R.-P. Happonen, P.~J\"a\"asaari, , A.~Kivel\"a,
  J.~Kuortti, J.~M. Luukinen, J.~Malinen, T.~Murtola, R.~Parkkola,
  J.~Saunavaara, and M.~Vainio.
\newblock Large scale data acquisition of simultaneous MRI and speech.
\newblock {\em Appl. Acoust.}, 83(1):64--75, 2014.

\bibitem{A-B-P:VDSEES}
E.~Acerbi, G.~Buttazzo, and D.~Percivale.
\newblock A variational definition of the strain energy for an elastic string.
\newblock {\em J. Elasticity}, 25:137--148, 1991.

\bibitem{A-A-F-M:DJPMVM}
S.~Alessandrini, D.~Arnold, R.~Falk, and A.~Madureira.
\newblock Derivation and justification of plate models by variational methods.
\newblock In M.~Fortin, editor, {\em Plates and shells (Qu\'ebec, QC, 1996)},
  volume~21 of {\em CRM Proc. Lecture Notes}, pages 1--20. Amer. Math. Soc.,
  Providence, RI, 1999.

\bibitem{C-K:TVINAS}
T.~Chiba and M.~Kajiyama.
\newblock {\em The Vowel, Its Nature and Structure}.
\newblock Phonetic Society of Japan, 1958.

\bibitem{EE:CSWHE}
E.~Eisner.
\newblock Complete solutions of the "{W}ebster" horn equation.
\newblock {\em J. Acoust. Soc. Am.}, 41(4):1126--1146, 1967.

\bibitem{GF:ATSP}
G.~Fant.
\newblock {\em Acoustic Theory of Speech Production}.
\newblock Mouton, The Hague, 1960.

\bibitem{H-S:FDHL}
C.~Hanna and J.~Slepian.
\newblock The function and design of horns for loudspeakers (reprint).
\newblock {\em J. Audio Eng. Soc}, 25(9):573--585, 1977.

\bibitem{H-L-M-P:WFWE}
A.~Hannukainen, T.~Lukkari, J.~Malinen, and P.~Palo.
\newblock Vowel formants from the wave equation.
\newblock {\em J. Acoust. Soc. Am.}, 122(1):EL1--EL7, 2007.

\bibitem{KZ:CSMSCG}
P.~Kuchment and H.~Zeng.
\newblock Convergence of spectra of mesoscopic systems collapsing onto a graph.
\newblock {\em J. Math. Anal. Appl.},
  258(2):671--700, 2001.

\bibitem{L-L:AMAEMAI}
M.~Lesser and J.~Lewis.
\newblock Applications of matched asymptotic expansion methods to acoustics.
  {I}. {Th}e {W}ebster horn equation and the stepped duct.
\newblock {\em J. Acoust. Soc. Am.}, 51(5):1664--1669, 1971.

\bibitem{L-L:AMAEMAII}
M.~Lesser and J.~Lewis.
\newblock Applications of matched asymptotic expansion methods to acoustics.
  {II}. {T}he open-ended duct.
\newblock {\em J. Acoust. Soc. Am.}, 52(5):1406--1410, 1972.

\bibitem{L-M:WECAD}
T.~Lukkari and J.~Malinen.
\newblock {W}ebster's equation with curvature and dissipation.
\newblock arXiv:1204.4075, 2013.
\newblock Submitted.

\bibitem{M-S-W:HTCCS}
J.~Malinen, O.~Staffans, and G.~Weiss.
\newblock When is a linear system conservative?
\newblock {\em Q. Appl. Math.}, 64(1):61--91, 2006.

\bibitem{M-S:CBCS}
J.~Malinen and O.~Staffans.
\newblock Conservative boundary control systems.
\newblock {\em J. Differential Equations}, 231(1):290--312, 2006.

\bibitem{T-M:MVP}
T.~Murtola.
\newblock {\em Modelling vowel production}.
\newblock Licentiate thesis, Aalto University, 2014.

\bibitem{AN:PM}
A.~Nayfeh.
\newblock {\em Perturbation Methods}.
\newblock Wiley-Interscience, New York, 1973.

\bibitem{N-T:APDVCS}
A.~Nayfeh and D.~Telionis.
\newblock Acoustic propagation in ducts with varying cross sections.
\newblock {\em J. Acoust. Soc. Am.}, 54(6):1654--1661, 1973.

\bibitem{SR:STSVCALDF}
S.~Rienstra.
\newblock Sound transmission in slowly varying circular and annular lined ducts
  with flow.
\newblock {\em J. Fluid Mech.}, 380:279--296, 1999.

\bibitem{SR:WHER}
S.~Rienstra.
\newblock {W}ebster's horn equation revisited.
\newblock {\em {SIAM} J. Appl. Math.}, 65(6):1981--2004, 2005.

\bibitem{R-E:NCBMSFESSPLFD}
S.~Rienstra and W.~Eversman.
\newblock A numerical comparison between the multiple-scales and finite-element
  solution for sound propagation in lined flow ducts.
\newblock {\em J. Fluid Mech.}, 437:367--384, 2001.

\bibitem{R-H:IA}
S.~Rienstra and A.~Hirschberg.
\newblock {\em An introduction to acoustics}.
\newblock Downloadable from http://www.win.tue.nl/~sjoerdr/papers/boek.pdf,
  2013.


\bibitem{RS:STRIP}
J.~Rubinstein and M.~Schatzman.
\newblock Variational problems on multiply connected thin strips I: Basic
  estimates and convergence of the laplacian spectrum.
\newblock {\em Arch. Ration. Mech. An.}, 160(4):271--308,
  2001.

\bibitem{VS:GPWHT}
V.~Salmon.
\newblock Generalized plane wave horn theory.
\newblock {\em J. Acoust. Soc. Am.}, 17(3):199--211, 1946.

\bibitem{VS:NFH}
V.~Salmon.
\newblock A new family of horns.
\newblock {\em J. Acoust. Soc. Am}, 17(3):212--218, 1946.

\bibitem{OS:WPLS}
O.~Staffans.
\newblock {\em Well-Posed Linear Systems}.
\newblock Cambridge University Press, Cambridge, 2004.

\bibitem{AW:AITHP}
A.~Webster.
\newblock Acoustic impedance, and the theory of horns and of the phonograph.
\newblock {\em Proc. Natl. Acad. Sci. USA}, 5:275--282, 1919.

\end{thebibliography}

\end{document}